\documentclass[11pt]{article}

\usepackage{amsfonts}
\usepackage{amssymb,amsmath,amsthm, enumerate}
\usepackage{latexsym}
\usepackage{fullpage}
\usepackage{hyperref}

\usepackage{color}

\usepackage{graphicx}

\newtheorem{theorem}{Theorem}[section]

\newtheorem{prop}[theorem]{Proposition}

\newtheorem{lemma}[theorem]{Lemma}
\newtheorem{cor}[theorem]{Corollary}

\theoremstyle{definition}
\newtheorem{remark}[theorem]{Remark}

\newcounter{tenumerate}

\def\P{\mathbb{P}}

\def \a {\alpha}
\def \b {\beta}

\def \BD {{\cal BD}}

\def \D {\Delta}

\def \e {\varepsilon}
\def \END {{\rm END}}
\def \d {\delta}
\def \E {{\mathbb{E}}}
\def \F {{\cal F}}
\def \G {{\cal G}}

\def \k {\kappa}
\def \l {{\cal L}}

\def \P {{\mathbb{P}}}
\def \p {{\cal P}}
\def \R {\mathbb{R}}

\def \s {\sigma}
\def \SS {{\cal S}}
\def \SL {{\cal SL}}
\def \t {{\cal T}}

\def \Var {{\rm Var}}
\def \Z {{\mathbb{Z}}}
\def \( {\left( }
\def\) {\right) }
\def\[ {\left[}
\def\]{\right]}

\def\qed{\hfill $\Box$}

\begin{document}

\title{Liouville first passage percolation: geodesic length exponent is strictly larger than 1 at high temperatures}

\author{ Jian Ding\thanks{Partially supported by an NSF grant DMS-1455049, an Alfred Sloan fellowship, and NSF of China 11628101.} \\ University of Chicago \and Fuxi Zhang\thanks{Supported by NSF of China 11371040 and 11771027.}  \\
Peking University
}
\date{}

\maketitle

\begin{abstract}

Let $\{\eta(v): v\in V_N\}$ be a discrete Gaussian
free field in a two-dimensional box $V_N$ of side length
$N$ with Dirichlet boundary conditions. We study the Liouville first
passage percolation, i.e., the shortest path metric where each vertex is given a weight of $e^{\gamma \eta(v)}$ for some $\gamma>0$. We show that for sufficiently small but fixed
$\gamma>0$, with probability tending to $1$ as $N\to \infty$, all geodesics between vertices of macroscopic Euclidean distances simultaneously have (the conjecturally unique) length exponent strictly larger than 1.
\end{abstract}

\section{Introduction}
Let $B \subseteq \Z^2$ be finite and non-empty. The discrete Gaussian free field (DGFF) $\{\eta^B (v) : v\in B\}$ with Dirichlet boundary conditions is a mean-zero Gaussian process with
$$\E \eta^{B} (x) \eta^B (y) = \E_x \sum_{t = 0}^{\tau - 1} 1_{\{ S_t = y \}} \ \ \mbox{ for all  }x, y \in B\,,$$
where $\{ S_t : t = 0,1,2, \cdots \}$ is a simple random walk starting from $x$, and $\tau$ is the hitting time to the boundary $\partial B = \{ z \in B : \exists \ w \in B^c \mbox{ such that } z \mbox{ is a neighbor of } w \}$.

Let $V_N = [0, N]^2 \cap \Z^2$. We set $V_{3N} = [-N, 2N]^2 \cap \Z^2$, and believe that there is no ambiguity since $[0, 3N]^2 \cap \Z^2$ will not be used.
A path $P$ in $V_N$ is a sequence of vertices $v_0, v_1, \cdots, v_d$ in $V_N$, where $v_i$ is a neighbor of $v_{i+1}$ for all $i$. The weight of $P$ is defined to be
$w (P)  : = \sum_{z \in P} \exp \{ \gamma \eta^{V_{3N}} (z) \}$,  where $\gamma > 0$ plays the role of inverse-temperature. For $x, y\in V_N$, the Liouville first passage percolation (FPP) distance between $x$ and $y$ is defined to be $\min_{_P} w(P)$, where the minimization is taken over all paths in $V_N$ joining $x$ and $y$. The (unique with probability 1) minimizer is defined to be the geodesic between $x$ and $y$, and denoted by ${\rm Geo}_{N,x,y}$. Let $\| x-y \|$ be the Euclidean distance of $x$ and $y$, and $|A|$ be the cardinality of a finite set $A \subseteq \Z^2$. Our main result is on the (conjecturally unique) length exponent of such geodesics.
 \begin{theorem} \label{Thm.dimension}
There exists $\gamma_0>0$ such that the following holds. For each $\gamma \in (0, \gamma_0)$, there exists $\a = \a (\gamma ) > 0$ such that for every $\k \in (0,1)$,
 $$
\lim_{N \to \infty} \P \( \big| {\rm Geo}_{N,x,y} \big| > N^{1 + \a} \ \ \ \mbox{ for all } x,y \in V_N \mbox{ with } \| x-y \| \ge \k N \)  = 1\,.
 $$
 \end{theorem}
Theorem~\ref{Thm.dimension} states that with probability tending to 1 as $N\to\infty$, all geodesics simultaneously have length exponent strictly larger than 1. We note that Benjamini \cite{Benjamini10} asked the question on the dimension, which is analogous to the length exponent here, of the (conjecturally well-defined) scaling limit of geodesics. Indeed, in  \cite{Benjamini10}  it was suspected that this dimension is strictly larger than 1.

Theorem~\ref{Thm.dimension} follows from a combination of results in \cite{DD16, DG16} and the following theorem (which may be of independent interest). Let $\| P \| = \| x - y \|$ if $P$ is a path from $x$ to $y$. Set
 $$
\p_{\k, \a} =  \big\{ P : P \mbox{ is a path in } V_N, \ \| P \| \ge \k N,  \mbox{ and } |P| \le N^{1 + \a } \big\}  .
 $$

  \begin{theorem} \label{mainthm}
For each $\d \in (0,1)$, there exists $\a = \a (\d) > 0$ such that for every $\k \in (0,1)$,
 $$
\lim_{N \to \infty} \P \( \left| \left\{ z \in P : \eta^{V_{3N}} (z) \ge -  15 \sqrt \d \log N \right\} \right| \ge \frac 1 8 \k N \ \ \mbox{ for all } P \in \p_{\k, \a } \) = 1 .
 $$
 \end{theorem}
By Theorem~\ref{mainthm}, with probability tending to 1 as $N\to \infty$, we have that
 \begin{equation}\label{eq-thm-consequence}
w (P) \ge \frac 1 8 \k N e^{- 15 \gamma \sqrt \d \log N} \ge N^{1 - 16 \gamma  \sqrt \d} \ \  \mbox{ for all } P \in \p_{\k, \a }\,.
 \end{equation}
By \cite[Theorem 1.3]{DG16}, we have that (for $\gamma < \gamma_0$)   as $N\to \infty$,
 $$
\max_{x, y\in V_N} \E \big( w ({\rm Geo}_{N,x,y} ) \big) \le N^{1 - \gamma^{4/3} / (\log \gamma^{-1})^2}\,.
 $$
By \cite[Theorem 3.1 and Proposition 6.7]{DD16} (where we use \cite[Theorem 3.1]{DD16} to verify assumptions in \cite[Proposition 6.7]{DD16}), we have that for $\gamma < \gamma_0$ and every $\epsilon>0$, there exists $C = C_\epsilon>0$ such that
 $$
\max_{x, y\in V_N} w ({\rm Geo}_{N,x,y} )\leq C  \max_{x, y\in V_N} \E \big( w ({\rm Geo}_{N,x,y} ) \big) \,
 $$
with probability at least $1-\epsilon$. The main result of \cite{DD16} is  the existence of sub-sequential scaling limit, and as an intermediate step it was also proved that the diameter has the same order as the distance between two fixed vertices (as stated above).
  Combining the last two displays, we get that with probability tending to 1
as $N\to \infty$,
 $$
\max_{x, y\in V_N} w ({\rm Geo}_{N,x,y} ) \le N^{1 - \gamma^{4/3} / (\log \gamma^{-1})^3}\,.
 $$
Setting $\d$ sufficiently small depending on $\gamma$ (e.g., $\d  < \gamma$) and combining the preceding inequality with \eqref{eq-thm-consequence}, we conclude that ${\rm Geo}_{N,x,y}  \notin \p_{\k, \a }$, establishing Theorem~\ref{Thm.dimension}. Thus, the main task of the present paper is to prove Theorem~\ref{mainthm}.

 \begin{remark}
In fact, for every (not necessarily small) $\gamma > 0$, if one can show that the distance exponent for the Liouville FPP is strictly less than 1, then combined with Theorem 1.2 this will yield that geodesics have length exponent strictly larger than 1.
 \end{remark}

 \begin{remark}
If one considers the DGFF in $V_N$, Theorem~\ref{Thm.dimension} and Theorem~\ref{mainthm} also hold, with an additional assumption that $x$, $y$ are away from boundary, i.e. $\| x -z \|$, $\| y- z \| \ge a N$ for all $z \in \partial V_N$, where $a \in (0, 1)$ and is fixed. The proof is essentially the same. We choose to consider the DGFF in $V_{3N}$ to avoid unnecessary cumbersome notation.
  \end{remark}
 \begin{remark} \label{Rem.openquestions}
The proof of Theorem~\ref{mainthm} should work for other log-correlated Gaussian fields with $\star$-scale invariant kernels as in \cite{DRSV14}. In particular, our method should be adaptable for proving an analogue of Theorem~\ref{mainthm} for every reasonable discrete approximation of the \emph{continuous} Gaussian free field. There are two related open problems which are not considered in the present paper: (i) as the discrete approximation gets finer, whether geodesics converge to those of the continuous GFF; (ii) whether geodesics of the continuous GFF have dimension strictly larger than 1.
 \end{remark}

\subsection{Related works}

Recently, there have been a few works on the Liouville first passage percolation metric \cite{DD16, DG16}. As mentioned above, in \cite{DG16} an upper bound on the distance exponent was derived, and in \cite{DD16} a sub-sequential scaling limit was proved for the normalized metric. Our work addresses another important aspect of this random metric. That is, geodesics under this random metric have fractal structures, in the sense that they have length exponent strictly larger than 1. (In contrast, for the classical FPP with i.i.d.\ weights, assuming say the distribution of the weight is continuous, geodesics have length exponent 1 \cite[Theorem~4.6]{ADH17}.) This, in turn, emphasizes the fractal nature of the Liouville FPP,  which is drastically different from the classical FPP.

In \cite{DL16}, the chemical distance (i.e., the graph distance in the induced open cluster) for the percolation of level sets of the two-dimensional DGFF was studied. In particular, \cite[Theorem 1.1]{DL16} implies that there exists a path of length exponent 1 and Liouville FPP weight $O \( N^{1+o(1)} \) $ joining the two boundaries of an annulus. This can be regarded as a complement of our Theorem~\ref{mainthm}. In \cite{DZ15}, a non-universality result was proved on the distance exponent for geodesics among log-correlated Gaussian fields. This suggests that the distance exponent is subtle. In light of \cite{DZ15}, there seems to be no reason to expect that the geodesic length exponent is universal among log-correlated fields.

Another related work is \cite{GHS16}, where the authors studied a type of discrete metric associated with the Liouville quantum gravity (LQG), and gave some bounds on the distance exponent of corresponding geodesics. While the Liouville FPP metric is expected to be related to random metric on the LQG, we refrain ourselves from an extensive discussion on the LQG metric. An interested reader is referred to \cite{MS15b,MS16} (as well as references therein) for a body of recent works on the construction of a metric in the continuum as well as its connection to the Brownian map. However, so far we see no mathematical connection between our work and \cite{GHS16,MS15b,MS16}. Finally, we remark that in a recent work \cite{LW16} the authors studied a random pseudo-metric on a graph defined via the zero-set of the Gaussian free field on its metric graph.

\medskip

\subsection{Notation convention}
 For  $z = (z_1, z_2), \ w = (w_1, w_2) \in \R^2$, let
 $$
 |z - w|_\infty = |z_1 - w_1| \vee |z_2 - w_2| \ \ \ \mbox{ and } \ \ \  \| z- w \| = \sqrt{ |z_1 - w_1|^2 + |z_2 - w_2|^2 }
 $$
be the $L^\infty$-distance and the $L^2$-distance of $z$ and $w$, respectively. Let $d (z, B) = \inf_{w \in B} \| z- w \| $. Similarly, we have $d (B_1, B_2)$, $d_\infty (z, B)$ and $d_\infty (B_1, B_2)$.

By side length of a box, we mean the $L^2$-distance of its two lower corners. For $x = (x_1, x_2) \in \Z^2$ and $\ell \in 2 \Z_+$, let
 $$
B_\ell (x) := \big( [x_1 - \ell / 2, x_1 + \ell/2] \times [x_2 - \ell / 2, x_2 + \ell / 2] \big) \cap \Z^2
 $$
be the box centered at $x$ and of side length $\ell$ (with the convention that $B_0 (x) = \{ x \}$). For $x \in \R^2$ and $\ell > 0$, let
 $$
B(x, \ell) = \{ z \in \R^2 : \| x - z \| \le \ell \}
 $$
be the $\ell_2$-ball centered at $x$ and of radius $\ell$.

For an integer $r \ge 1$, let $[r] = \{ 1, \ldots, r \}$. Let $\{ S_t : t  =  0,1,2, \cdots \}$ be a simple random walk on $\Z^2$, and $\tau_D$ be the time it hits $D$.

\subsection{Convention for constants and parameters}

Throughout this paper, let $C_0, C_1, C_2, \ldots > 0$ be universal constants. Let $k \in \Z_+$ and $K = 2^k$ be fixed and large. Here and in what follows, by $K$ is large we mean that it is large with respect to universal constants, so that we can assume that the inequalities such as $C_2 \le C_1 k$ and $ K^{\frac 1 {K^2 k}} e^{ - \frac 1 {K^2+1} } < 1$ hold.

Let $\d, \k \in (0,1)$ be as in Theorem~\ref{mainthm}, and $\e > 0$. Let $K (\d, \e)$ be a fixed and large integer relying on $\e$ and $\d$, and of the form $2^k$. Similarly, we have $K_1 (\e), K_2 (\e), K_3 (\d)$.

Recall that $V_N = [0, N]^2 \cap \Z^2$ and $V_{3N} = [-N, 2 N]^2 \cap \Z^2$. Let $V_{2N} = \left[ - \frac 1 2 N, \frac 3 2 N \right]^2 \cap \Z^2$. We will consider the limiting behavior when $N \to \infty$. Set $m \in \Z_+$ such that
 $$
K^{m+1} \le \k N < K^{m+2}.
 $$
Note that $m \to \infty$ since $N \to \infty$ and $K, \k$ are fixed. Thus, we can assume without loss that the inequalities such as $(K e^{-k})^{\d m} < \d$ hold.

\subsection{Outline of the proof}

The general proof strategy in this paper is multi-scale analysis, which has seen powerful applications in the percolation theory (see, e.g., \cite{CCD88, DM90, Chayes95, Chayes96, Orz98, AB99}). In particular, our proof of Theorem~\ref{mainthm} is inspired by the methods employed in \cite{Chayes96, Orz98} (especially, \cite{Orz98}) in the study of the fractal percolation process. A particular instance of the fractal percolation process is as follows. Partition $V = [0,1]^2$ into identical boxes $\{ B_{r, i} \}$ of side length $2^{-r}$. Remove each $B_{r,i}$, $r \ge 1$, $i \in [2^{2r}]$ independently with probability $p$. The vertices which are removed (that is, at least one of the boxes containing the vertex is removed) are said to be closed, and those retained are said to be open. In the supercritical case, it was known from \cite{CCD88} that with non-vanishing probability there exists an open crossing, where by crossing we mean a path  connecting the left and right boundaries of $V$.  In \cite{Chayes96, Orz98} it was shown that each such open crossing has (box-counting) dimension strictly larger than 1. Roughly speaking, the proof strategy of \cite{Orz98} is that in every scale $2^{-r}$, a sufficient number of boxes in $\{ B_{r,i}, i \in [2^{2r}] \}$ are removed, which forces each open crossing to take detours in this scale with a non-vanishing frequency. Thus, each open crossing has dimension strictly larger than 1.

The framework of our proof is similar in spirit to that of \cite{Orz98} since the two-dimensional DGFF has a similar hierarchical structure (see Section~\ref{Sect.hiearchy}). Recall $K^{m+1} \le \k N < K^{m+2}$. We will show that for each $x\in V_N$ we have $\eta^{V_{3N}} (x) \approx \sum_{j=0}^{m-1} \eta_j (x)$, where $\eta_0(x), \ldots, \eta_{m-1}(x)$ are mutually independent and of mean zero.  In addition, we have that $\E \eta_j (x)^2 = O(k)$, and that $\eta_j (x)$ and $\eta_j (y)$ are independent if $|x-y| \ge K^{j+1}$. However, there are two important differences from the fractal percolation process, as follows.
 \begin{itemize}
\item In the fractal percolation process, a vertex $z$ is removed as long as one removes the box in $\big\{ B_{r,i}, i \in [2^{2r}] \big\}$ containing $z$ for some $r$. However, in our context, one may naturally define a vertex $z$ to be removed if $\eta^{V_{3N}} (z) \ge - 15 \sqrt{\d} \log N$ (15 here is a number that is chosen somewhat arbitrarily for convenience of exposition later; similar for $\tfrac18$ below). Thus, whether a vertex is removed or not can only be verified by putting together the values $\eta_j (z)$'s for (almost) all $j$'s, rather than by the value of  $\eta_j (z)$ for some $j$. This is a huge difference between the fractal percolation process and our model.

\item  In the fractal percolation process, one only needs to demonstrate that for each crossing with dimension close to $1$, there exists one vertex on it which is removed. However, in our context, we must show that for each $P \in \p_{\k, \a}$, there exist on $P$ a large number (namely, at least $\frac 1 8 \k N$) of vertices which are all removed.
\end{itemize}

In order to address the aforementioned difficulties, we will associate each path $P$ in
 $$
\p_{\k} : = \{ P :  \| P \| \ge \k N \}
 $$
with a tree $\t_P$ via a deterministic construction, as demonstrated in Section~\ref{Sect.sub-paths}. The nodes of $\t_P$ correspond to a family of sub-paths of $P$, where the parent/child relation in $\t_P$ corresponds to path/sub-path relation in the plane. In particular, the root corresponds to $P$, the leaves correspond to vertices on $P$, and a node $u$ in the $r$-th level of $\t_P$ corresponds to a sub-path in scale $K^{m-r}$ (see \eqref{Eq.SLj} below for definition). We will define the tame/untamed property of a sub-path (see \eqref{Eq.tildeExy} below for definition), equivalently, of a node. Roughly speaking, a sub-path is tame if it lies in an ellipse whose focuses are its endpoints and whose ratio of width to height is $K$ to 1. Thus, it behaves more or less like a straight segment in the corresponding scale (recall that $K$ is large).

With the above definitions at hand, we will show that tame nodes have degree $K$, while untamed nodes have degree $\ge K$ always and have degree $\ge K+1$ at non-vanishing frequency on every branch of the tree (see Proposition~\ref{Prop.sub-paths}). In addition, note that for $P$ with length exponent close to $1$, the number of leaves in $\t_P$ is small. Thus, untamed nodes should be rare in $\t_P$. But it turns out that counting measure is not a good measurement of rareness for our purpose. In order to be compatible with our analysis later,  in Section~\ref{Sect.Lemma2}  we will consider the uniform flow $\theta_P$ on $\t_P$ instead, and show in Proposition~\ref{Lem.Lemma2} that the fraction of flow supported on untamed nodes is small.

In Section~\ref{Sect.Lemma1}, we will study the behavior of a tame sub-path. For all tame sub-paths with fixed endpoints, in Proposition~\ref{Prop.Lemma1} we will derive a uniform upper bound on the fraction of its \emph{open} children, i.e., the children with $\eta_{j-1} \ge \e k$ (see the beginning of Section~\ref{Sect.proofs} for a rigorous definition for open sub-paths). The proof of Proposition~\ref{Prop.Lemma1} crucially relies on the fact that all these tame sub-paths lie in the same thin tunnel in the corresponding scale (in other words, they are essentially one-dimensional in that scale).

In Section~\ref{Sect.Lemma3}, based on Proposition~\ref{Lem.Lemma2} and Proposition~\ref{Prop.Lemma1} we will prove Proposition~\ref{Lem.Lemma3}, which bounds the fraction (measured with respect to the uniform flow $\theta_P$) of open nodes for all possible paths simultaneously. Thanks to Proposition~\ref{Lem.Lemma2}, we can relatively easily bound the fraction of open nodes with untamed parents, and thus the main challenge for proving Proposition~\ref{Lem.Lemma3} lies on controlling the fraction of open nodes with tame parents. More precisely, we only need to show that simultaneously for all paths in $\p_{\k, \a}$, a typical tame node $u$ in scale $K^j$ contributes a small fraction of open children. In order for a simultaneous control, we will apply a union bound in every scale in an inductive manner, via a multi-scale analysis on corresponding tree structures. Our strategy is in flavor similar to the chaining argument (see \cite{Talagrand14} for an excellent account on this topic), which is a powerful method in bounding the maximum of a random process. The success in our context crucially relies on the introduction of the uniform flow $\theta_P$ as we describe next. Suppose that a node has a large degree $d$. On the one hand, the choices of locations of the children grow exponentially in $d$. On the other hand, (in light of the definition of uniform flow) we will perform an average over $d$ children, and thus obtain a large deviation. In addition, the probability decays exponentially in $d$, thereby beating the growth of enumeration if we choose our parameters appropriately. The key step for the induction analysis will be carried out in Lemma~\ref{Lem.inductionxir}.

Finally, in Section~\ref{Sect.Pfmain} we will show that for each $P\in \p_{\k, \a}$ there are at least $\frac 1 8 \k N$ vertices $z$ on $P$ such that $\sum_{j=0}^{m-1} \eta_j (z) \ge - m \e k = - \sqrt {\d} \cdot O(\log N)$ by setting $\e = O(\sqrt \d)$, completing the proof of Theorem~\ref{mainthm}. The estimates in Section~\ref{Sect.Pfmain} are more or less straightforward, and thus we do not discuss further here.

\medskip

\noindent {\bf Acknowledgement.} We warmly thank an anonymous referee for a detailed report, which leads to a significant improvement on exposition.

\section{Preliminaries on two-dimensional DGFF} \label{Sect.preliminary}

\subsection{Log-correlation} \label{Sect.logcorr}

Recall that $\{ S_t \}$ is a simple random walk on $\Z^2$, and the DGFF in $B$ has covariance given by the Green function associated with $\{ S_t \}$ in $B$. An explicit formula of the Green function is given in \cite[Proposition~4.6.2]{LL10}.
 \begin{equation}  \label{Eq.covariance}
\E \eta^B (x) \eta^B (y) = \E_x \sum_{t=0}^{\tau - 1} 1_{\{ S_t = y \}} = \sum_{z \in \partial B} \P_x (S_\tau = z)  a (z - y) - a (x - y)  ,
 \end{equation}
where $\tau = \tau_{\partial B}$,
 \begin{equation} \label{Eq.axax}
\left| a ( x ) - \( \frac 2 \pi \log \| x \| + \frac {2 \bar {\gamma} + \log 8}{\pi} \) \right| \le C_0 \| x \|^{-2}  \ \ \mbox{ with } a(0) = 0,
 \end{equation}
$\bar {\gamma}$ is the Euler constant, and $C_0 > 0$ is a universal constant \cite[Theorem~4.4.4]{LL10}.

Suppose that $B \subseteq \Z^2$ is a box of side length $\ell$. Set $B^{( \frac 1 {10} )} : = \{ z \in B : d_\infty (z, \partial B) > \frac 1 {10} \ell \}$ (here $\frac 1 {10}$ is a somewhat arbitrary choice). Then, one can check that the DGFF is log-correlated, i.e. that there are universal constants $C_1, C_2 > 0$ such that
 \begin{equation} \label{Eq.logcorrelated}
\left| \E \eta^B (x) \eta^B (y)  - C_1 \log_2 \frac {\ell}{|x-y|_\infty \vee 1} \right| \le C_2 \ \ \mbox{ for all } x, y \in B^{(\frac 1{10})} .
 \end{equation}
 \begin{remark} \label{Rem.aboutC1}
The condition in \eqref{Eq.logcorrelated} is a general assumption for log-correlated Gaussian fields. For the DGFF, $C_1 = \frac 2 {\pi} \log 2$ by \eqref{Eq.axax}.
 \end{remark}
\noindent Moreover, we can choose $C_2$ such that the following hold: For  $\ell \in 2 \Z_+$ and $x_0 \in \Z^2$, set $B = B_{\ell} (x_0)$ (recall that it is the box centered at $x_0$ and of side length $\ell$). Then,
 \begin{equation} \label{Eq.covatboundary}
\E \eta^B (x_0) \eta^B (y) \le C_2 \ \ \mbox{ for all } y \in B \setminus B^{(\frac 1 {10})} .
 \end{equation}
Note that there exists a universal constant $C_3 > 0$ such that for each box $B \subseteq \Z^2$ of side length $\ell$,
 \begin{equation} \label{Eq.Stau}
\P_x (S_\tau = z) \le C_3 \ell^{-1}  \ \ \mbox{for all } x \in B^{(\frac 1 {10})} \mbox{ and } z \in \partial B
 \end{equation}
(see, e.g., \cite[Proposition 8.1.4]{LL10}). Furthermore, for $\ell, \tilde \ell \in 2 \Z_+$ with $\tilde \ell > \ell$, we have
 \begin{equation} \label{Eq.variance}
\left| \E \eta^{B_{\tilde \ell} (x_0)} (x_0)^2 - \E \eta^{B_\ell (x_0)} (x_0)^2 \right| \le \frac {4 C_3 (\tilde \ell - \ell)} {\ell}.
 \end{equation}

For the sake of completeness, next we show \eqref{Eq.logcorrelated}, \eqref{Eq.covatboundary} and \eqref{Eq.variance}.

 \begin{proof}[Proof of \eqref{Eq.logcorrelated}, \eqref{Eq.covatboundary} and \eqref{Eq.variance}.]
Set $C_2 : = 2 C_0 + \log 15 +  \frac {2 \bar {\gamma} + \log 8}{\pi} $.  For \eqref{Eq.logcorrelated}, note that $\ell / 10  \le \| x-z \| \le  \sqrt 2 \ell$ for all $z \in \partial B$. Combined with \eqref{Eq.axax}, this implies that $\left| a(z-x) - \frac 2 \pi \log \ell \right| \le \left| \log \frac {\| x-z \| } \ell \right| + \frac {2 \bar {\gamma} + \log 8}{\pi}  + C_0 \le C_2$. Thus \eqref{Eq.logcorrelated} holds if $x = y$. If $x \neq y$, then $\frac {\ell / 10} { \sqrt 2 | x-y |_\infty} \le \frac {\| z - y \|}{\| x - y \|} \le \frac {\sqrt 2 \ell}{ | x-y |_\infty}$ for all $z \in \partial B$. Combined with \eqref{Eq.covariance} and \eqref{Eq.axax}, this yields that $\left| \big( a(z-y) - a(x-y) \big) -\frac 2 \pi \log \frac {\ell} {| x-y |_\infty} \right| \le \log 15 + 2 C_0 \le C_2$. Thus \eqref{Eq.logcorrelated} holds.

Set $B = B_\ell (x_0)$ and $\tilde B = B_{\tilde \ell} (x_0)$ for brevity. For \eqref{Eq.covatboundary}, since $\frac {\| z-y \|}{\| x_0-y \|} \le \frac {\sqrt 2 \ell}{(\frac 1 2 - \frac 1 {10}) \ell}$ for all $y \in B \setminus B^{(\frac 1 {10})}$ and $z \in \partial B$, one has $a(z - y) - a(x-y) \le 2C_0 + \frac 2 \pi \log \frac {\| z-y \|}{\| x_0 - y \|} \le 2 C_0 + \log 5$. Consequently, $\E \eta^B (x_0) \eta^B (y) \le C_2$, which is \eqref{Eq.covatboundary}. For \eqref{Eq.variance}, denote by $\tau$ and $\tilde \tau$ respectively the times $\{ S_t \}$ hits $\partial B$ and $\partial \tilde B$. Then,
$$\E \eta^{\tilde B} (x_0)^2 - \E \eta^B (x_0)^2 = \E_{x_0} \sum_{t = \tau}^{\tilde \tau - 1} 1_{\{ S_t = x_0 \} } = \sum_{z \in \partial B} \P_{x_0} (S_\tau = z) \E \eta^{\tilde B} (x_0) \eta^{\tilde B} (z) \le \frac {C_3} \ell \E_{x_0} \sum_{t=0}^{\tilde \tau - 1} 1_{\{ S_t \in \partial B \}},$$ where in the last inequality we have used \eqref{Eq.covariance} and \eqref{Eq.Stau}. In addition, it is not difficult to check that $\E_{x_0} \sum_{t=0}^{\tilde \tau - 1} 1_{\{ S_t \in \partial B \}} \le 4 (\tilde \ell - \ell )$ via an analysis of the process $\{ |S_t|_\infty \}$, a similar proof of which is referred to \cite[Lemma~2.2]{DL16}. Then, the proof of \eqref{Eq.variance} is completed.
 \end{proof}

\subsection{Hierarchical structure}  \label{Sect.hiearchy}

For $B \subseteq V_{3 N}$, denote by $H^B$ the Gaussian field satisfying (a) $H^B$ coincides with $\eta^{V_{3N}}$ on $B^c \cup \partial B$,  (b) $H^B$ is harmonic in $B \setminus \partial B$. That is, $H^B$ is the harmonic extension of $\eta^{V_{3N}} |_{B^c \cup \partial B}$. Set $\eta^B : = \eta^{V_{3N}} - H^B$. The well-known Markov property of the DGFF states that $\eta^B$ is a version of the DGFF in $B$ with Dirichlet boundary conditions, and it is independent of $H^B$. In words, $\eta^{V_{3N}} = \eta^B \bigoplus H^B$ is an orthogonal decomposition. This property also holds when $V_{3N}$ is replaced by any $D$ with $B \subseteq D$. Especially, the corresponding  orthogonal decomposition is $\eta^D = \eta^B \bigoplus (\eta^D - \eta^B)$, if $B \subseteq D \subseteq V_{3N}$. We are now ready to introduce the hierarchical structure of $\eta^{V_{3N}}$.

 \begin{figure}[h]
  \includegraphics[width=17cm]{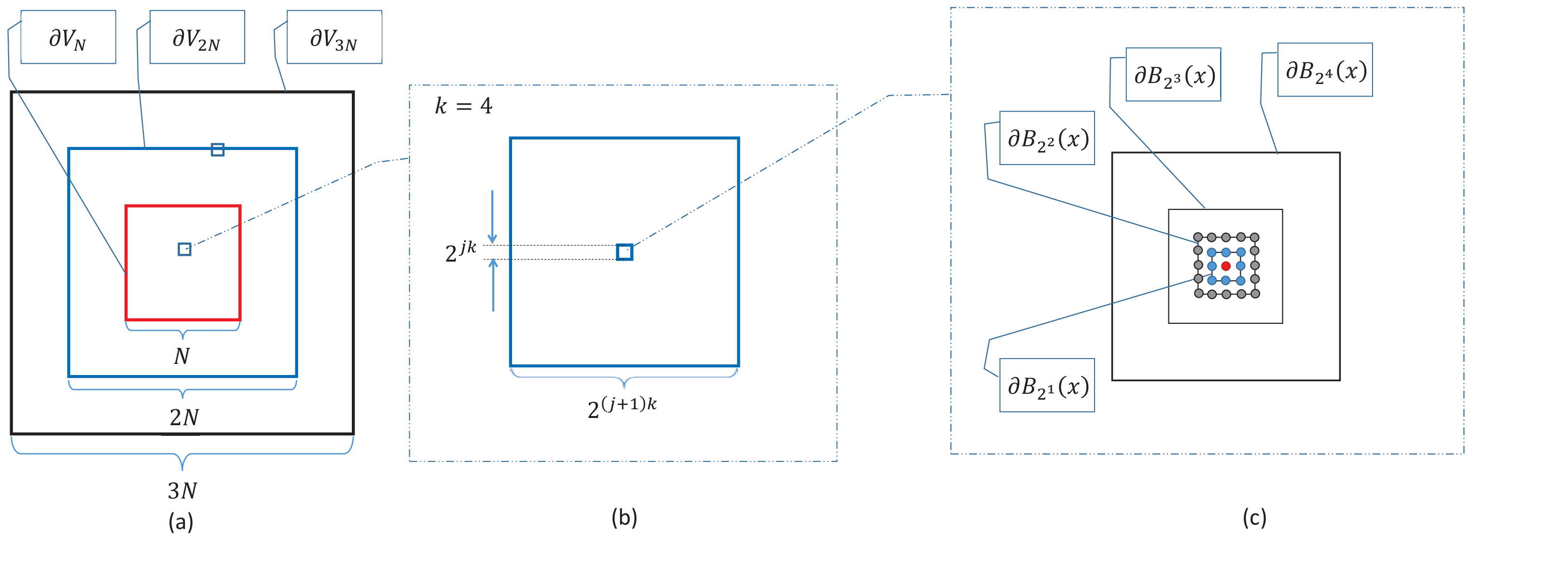}
\\ \vspace{-1cm} \caption{In Figure (a), the tiny boxes are $B_{K^m} (x)$'s for $x \in V_{2N}$. They are tiny and lie inside $V_{3N}^{(10)}$, since $K^m / N \le \k / K$. In Figure (b), the two boxes are samples of $B_{2^{(j+1)k}}$ and $B_{2^{jk}}$ respectively---note that the latter is a tiny sub-box in the former, which visually emphasizes the fact that $k$ is large. In Figure (c), the round points stand for vertices in $\Z^2$: the red one is $x$ and the blue ones are its ($\ell_\infty$-)neighbors.} \end{figure}

Recall that $K = 2^k$ is large but fixed, and $m$ is selected such that $K^{m+1} \le \k N < K^{m+2}$. Let $H_{2^r} (x) : = H^{B_{2^r} (x)} (x)$, $r \ge 1$ for brevity, and let $H_1 (x) := \eta^{V_{3 N}} (x)$. Define
 \begin{equation} \label{Eq.defetaj}
X_r (x) : = H_{2^r} (x) - H_{2^{r+1}} (x) \ \ \ \mbox{and} \ \ \ \eta_j (x) = \sum_{r = jk}^{(j+1) k - 1} X_r (x),
  \end{equation}
where $r = 0, 1, \ldots, m k - 1$ and $j=0,1, \ldots, m-1$. Then
 $$
\eta^{V_{3N}} (x) = \eta (x) + H_{K^m} (x) , \ \ \ \mbox{ where } \eta (x) := \sum_{j=0}^{m-1} \eta_j (x).
 $$
We will mainly study the term $\eta (x) = \sum_{j=0}^{m-1} \eta_j (x)$, while the term $H_{K^m} (x)$ will be controlled relatively easily.

By the Markov property, we have
  \begin{equation} \label{Eq.Xindpt}
\left\{ \begin{array}{l} X_r (x), \  r = 0, 1,  \cdots, mk-1 \mbox{ are mutually independent;} \\ X_r (x) \mbox{ and } X_r (y) \mbox{ are independent if } |x-y|_\infty \ge 2^{r+1}; \\ \eta_j (x) \mbox{ and } \eta_j (y) \mbox{ are independent if } |x-y|_\infty \ge K^{j+1}. \end{array} \right.
 \end{equation}
Furthermore, note $\sum_{i=0}^{r-1} X_i (x) = \eta^{B_{2^r} (x)} (x)$. By \eqref{Eq.logcorrelated} and \eqref{Eq.Xindpt}, we have
 \begin{equation} \label{Eq.C1C2}
\E X_r (x)^2 \le C_1 + 2 C_2 \ \ \ \mbox{and} \ \ \ |\E \eta_j (x)^2 - C_1 k| \le 2 C_2 .
 \end{equation}

Next, we will show that the covariances of all $\eta_j$'s are bounded from below. Recall $V_{2N} = \left[ -\frac 1 2 N, \frac 3 2 N \right]^2 \cap \Z^2$.
 \begin{lemma} \label{Lem.positivecorrelated}
There exists a universal constant $C_4 > 0$ such that $\E \eta_j (x) \eta_j (y) \ge - C_4$ for all $x,y \in V_{2N}$.
 \end{lemma}
 \begin{proof}
We only give the proof for $j \ge 1$, since that for $j = 0$ is similar by setting $B_1 = \{ x \}$ below. Set $B_1 = B_{K^j} (x)$ and $B_2 = B_{K^{j+1}} (x)$  for brevity. Recall $\eta_j (x) = \eta^{B_2} (x) - \eta^{B_1} (x)$. By the Markov property,
 \begin{equation} \label{Eq.corr}
\E \eta_j (x) \eta^{V_{3N}} (z) = \left\{ \begin{array}{ll} 0 & \mbox{ for all } z \in V_{3N} \setminus B_2 , \\ \E \eta^{B_2} (x) \eta^{B_2} (z) & \mbox{ for all } z \in B_2 \setminus B_1 ,  \\ \E \eta^{B_2} (x) \eta^{B_2} (z) - \E \eta^{B_1} (x) \eta^{B_1} (z) & \mbox{ for all } z \in B_1 . \end{array} \right.
 \end{equation}
In the case $z \in B_1$, one has $\E \eta^{B_2} (x) \eta^{B_2} (z) - \E \eta^{B_1} (x) \eta^{B_1} (z) = \E_x \sum_{t= \tau_1}^{\tau_2 - 1} 1_{\{ S_t = z \}}$
by \eqref{Eq.covariance}, where $\tau_1$ and $\tau_2$ are respectively the times $\{ S_t \}$ hits $\partial B_1$ and $\partial B_2$. Therefore, in all these three cases, we have
\begin{equation}\label{eq-positive-relation}
\E \eta_j (x) \eta^{V_{3N}}  (z) \ge 0 .
\end{equation}

Note $\eta_j (y) = H_{K^j} (y) - H_{K^{j+1}} (y)$. By \eqref{Eq.corr} and \eqref{eq-positive-relation},
 $$
\E \eta_j (x) \eta_j (y) \ge - \sum_{z \in \partial \tilde B_2} \P_y (S_\tau = z) \E \eta_j (x) \eta^{V_{3N}} (z) = - \sum_{z \in B_2 \cap \partial \tilde B_2} \P_y (S_\tau = z) \E \eta_j (x) \eta^{V_{3N}} (z) ,
 $$
where $\tilde B_2 = B_{K^{j+1}} (y)$ and $\tau = \tau_{\partial \tilde B_2}$. Note that $B_2 \cap \partial \tilde B_2$ lies in the union of a horizontal line and a vertical line in $B_2$. We will show that for every horizontal or vertical line $L$ in $B_2$,
 \begin{equation} \label{Eq.positivecorrelateionH}
 \sum_{z \in L } \E \eta_j (x) \eta^{V_{3N}} (z)  \le (3 C_1 + 2 C_2 + 1) K^{j+1}.
 \end{equation}
Assuming \eqref{Eq.positivecorrelateionH}, by \eqref{Eq.Stau} and \eqref{eq-positive-relation}, we have  $\E \eta_j (x) \eta_j (y) \ge - 2 (3 C_1 + 2 C_2 + 1) C_3$, completing the proof.

It remains to prove \eqref{Eq.positivecorrelateionH}. Without loss of generality, suppose $x = (0,0)$ and $L = \{ (z_1, z_2) \in B_2 : \ z_2 =b \}$, where  $b \in ( - \frac 1 2 K^{j+1}, \frac 1 2 K^{j+1} ) \cap \Z$.  We will prove  \eqref{Eq.positivecorrelateionH} by dividing $L$ into three parts (as follows) and bounding the contribution for each part separately.

Part 1. By \eqref{Eq.covatboundary} and \eqref{Eq.corr},
 $$
\sum_{z \in L \setminus B_2^{(\frac 1 {10})}} \E \eta_j (x) \eta^{V_{3N}} (z) \le \sum_{z \in L \setminus B_2^{(\frac 1 {10})}} \E \eta^{B_2} (x) \eta^{B_2} (z) \le C_2 \big| L \setminus B^{(\frac 1 {10})} \big| \le C_2 K^{j+1} .
 $$

Part 2. Let $L_0 = \{ (z_1, z_2) \in L : z_ 1 \in  [- \frac 1 2 K^j , \frac 1 2 K^j ] \cap \Z \}$. By \eqref{Eq.logcorrelated} and \eqref{Eq.corr},
 $$
\sum_{z \in ( L \setminus L_0 ) \cap B_2^{(\frac 1 {10}) } }  \E \eta_j (x) \eta^{V_{3N}} (z)   \le 2 \sum_{r = K^j / 2 + 1}^{K^{j+1} / 2} \( C_1 \log_2 \frac {K^{j+1}}{r} + C_2 \)  \nonumber
 \le
(3 C_1 + C_2) K^{j+1}.
 $$

Part 3. Suppose $z \in L_0 \cap B_2^{(\frac 1 {10})}$. If $|b| \ge K^j/2$, one has  $|x-z|_\infty \ge K^j / 2$. Otherwise, we have $z \in B_1$. By \eqref{Eq.corr},
 $$
\E \eta_j (x) \eta^{V_{3N} } (z)
 =
\E_z \sum_{t=\tau_1}^{\tau_2 -1} 1_{\{ S_t = x \}}  \le \max_{w \in \partial B_1} \E_w  \sum_{t=0}^{\tau_2 -1} 1_{\{ S_t = x \}}  = \max_{w \in \partial B_1} \E_w  \eta^{B_2} (x) \eta^{B_2} (w) ,
 $$
where $\tau_1$ and $\tau_2$ are respectively the times $\{ S_t  \}$ hits $\partial B_1$ and $\partial B_2$. Note that $|w - x|_\infty = K^j/2$ for all $w \in \partial B_1$. By \eqref{Eq.logcorrelated}, in both cases we have $\E \eta_j (x) \eta^{V_{3N} } (z)  \le C_1 \log_2 \frac {K^{j+1}}{K^j/2} + C_2$. Therefore,
 $$
\sum_{z \in L_0 \cap B_2^{(\frac 1 {10})} } \E \eta_j (x) \eta^{V_{3N} } (z)  \le (K^j+1) \times \big( C_1 \log_2 (2K) + C_2 \big) \le K^{j+1} .
 $$

Summing the upper bounds in these three parts, we conclude \eqref{Eq.positivecorrelateionH}.
 \end{proof}

 \begin{lemma} \label{Lem.Harmdiff}
There is a universal constant $C_5 > 0$ such that for $j \ge 0$ and $u,v \in V_{2N}$ satisfying $|u - v|_\infty \le K^j$, both $\E (H_{K^j} (u) - H_{K^j} (v))^2$ and $\E ( \eta_j  (u) - \eta_j (v) )^2$ are no more than $C_5 \frac {|u-v|_\infty }{K^j} $.
 \end{lemma}
 \begin{proof}
Suppose $u \neq v$ without loss. Set $C = 24 C_3 (1 + C_0) + 4 C_2$. We will show that for $j = 0, 1, \ldots, m$,
 \begin{equation} \label{Eq.boundHj1}
\E (H_{K^j}  (u) - H_{K^j} (v))^2 \le C |u-v|_\infty / K^j
 \end{equation}
if $|u - v|_\infty \le K^j$. Assuming this, we will complete the proof by setting $C_5 := 4 C$, noting that $\eta_j = H_{K^j} - H_{K^{j+1}}$.

Next, we will show \eqref{Eq.boundHj1}. For $j = 0$, note that $H_{K^0} = \eta^{V_{3N}}$. Then, \eqref{Eq.boundHj1} holds by \eqref{Eq.logcorrelated}. For $j \in [m]$, denote $B_{K^j} (u)$, $B_{K^j} (v)$, $B_{K^j + 2 |u-v|_\infty} (u)$ and $B_{K^j + 2 |u-v|_\infty } (v)$ by $B_u$, $B_v$, $\tilde B_u$ and $\tilde B_v$ for brevity. Without loss of generality, we suppose that $u_1 < v_1$ and $u_2 < v_2$, where $u = (u_1, u_2)$ and $v = (v_1, v_2)$. Let $B$ be the box of side length $K^j + |u-v|_\infty$ and with lower left corner $(u_1 - K^j / 2, u_2 -K^j / 2 )$. Then $B_u \cup B_v \subseteq B \subseteq \tilde B_u \cap \tilde B_v \subseteq V_{3N}$.

\begin{figure}[h!]
\hspace{3cm}  \includegraphics[width=8cm]{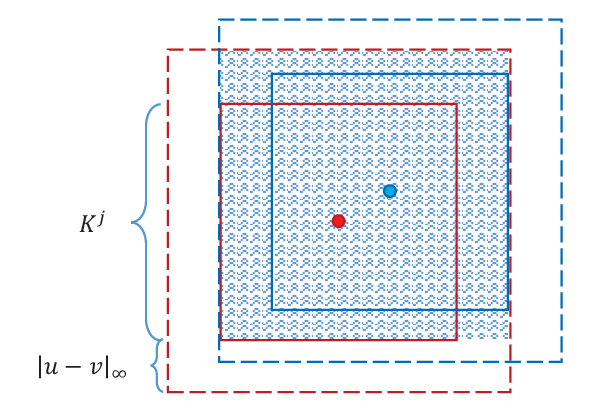}
\\ \vspace{-1cm} \caption{The round points are $u$ and $v$. The boxes with solid lines are $B_u$ and $B_v$. The boxes with dashed lines are $\tilde B_u$ and $\tilde B_v$. The box with shadow is $B$.} \end{figure}

Note that
 \begin{equation} \label{Eq.boundHj2}
H_{K^j}  (u) - H_{K^j} (v) = \big( H^{B_u}  (u) - H^{B} (u) \big) - \big( H^{B_v} (v) - H^B (v) \big) + \big( H^B (u) - H^B (v) \big) .
 \end{equation}
For the first two terms, note
 $$
\E \(   H^{B_u}  (u) - H^{B} (u)  \) ^2 = \E  \eta^{B} (u)^2 - \E \eta^{B_u} (u)^2 \le \E \eta^{\tilde B_u} (u)^2 - \E \eta^{B_u} (u)^2 \le \frac {4 C_3 |u-v|_\infty} {K^j},
 $$
where in the last inequality we have used \eqref{Eq.variance}. One can have the same estimate for the second term, by the same reasoning.

For the third term, let
 $$
\varphi (D)  = \sum_{z \in \partial D}  \big( \P_u (S_\tau = z) - \P_v (S_\tau = z) \big) \big( a(z - u) - a (z - v) \big)  ,
 $$
where $\tau = \tau_{\partial D}$, for $D = B$ or $D = V_{3N}$. By \eqref{Eq.covariance} and the orthogonal decomposition $\eta^{V_{3N}} = \eta^B \bigoplus H^B$, we have
 \begin{equation}\label{Eq.boundHj3}
\E  (H^B (u) - H^B (v))^2 \le |\varphi (V_{3N})| + |\varphi (B)|\,,
 \end{equation}
where we refer to \cite[Lemma 3.10]{BDZ14} for a similar derivation.

Note $u, v \in D^{(\frac 1 {10})}$. Suppose $z \in \partial D$. By \eqref{Eq.Stau}
 $$
| \P_u (S_\tau = z) - \P_v (S_\tau = z) | \le \frac {4 C_3}{|\partial D|}.
 $$
By \eqref{Eq.axax}, we obtain that
 $$
a(z - u) - a (z - v) \le \frac 2 \pi \log \frac {\| z - u \|}{\| z - v \|} + \frac {2 C_0}{K^j} \le 2 \big( 1 + C_0 \big) \frac {|u-v|_\infty}{K^j},
 $$
where we have used that
 $$
\log \frac {\| z - u \|}{\| z - v \|} \le \log \( 1 + \frac {\| u - v \|}{\| z - v \|} \) \le \frac {\| u - v \|}{\| z - v \|} \le \frac {\sqrt 2 |u-v|_\infty}{K^j / 2} \le 3 \frac {|u-v|_\infty}{K^j}.
 $$
By symmetry, $\big| a(z - u) - a (z - v) \big| \le  2 (1 + C_0) \frac {|u-v|_\infty}{K^j}$. It thus follows that
 $$
\big| \varphi (D) \big| \le  8 C_3 (1 + C_0) \frac {|u-v|_\infty}{K^j}  \ \ \mbox{ for $ D = B$ and $D = V_{3N}$}.
 $$
Combined with \eqref{Eq.boundHj3}, this implies that the third term in \eqref{Eq.boundHj2} has variance no more than $16 C_3 (1 + C_0) \frac {|u-v|_\infty}{K^j}$. Therefore, \eqref{Eq.boundHj1} holds.
 \end{proof}

At the end of this section, we state two lemmas, which will be useful in the next section.

 \begin{lemma}\label{Lem.maxinbox}
(Dudley's inequality, \cite[Theorem 4.1]{A90}) Let $B \subseteq \mathbb{Z}^2$ be a box of side length $\ell$ and $\{ G_w : w \in B \}$ be a mean zero Gaussian field satisfying
 $$
\E (G_z - G_w)^2 \le |z-w|_\infty / \ell \ \ \mbox {for all } z,w \in B.
 $$
Then $\E \max_{w \in B} G_w \le C_6$, where $C_6 > 0$ is a universal constant.
 \end{lemma}
 \begin{lemma} \label{Lem.concentration}
(Borell--Tsirelson inequality, \cite[Theorem 7.1, Equation (7.4)]{L01}) Let $\{ G_z : z \in B \}$ be a Gaussian field on a finite index set $B$. Set $\s^2 = \max_{z \in B} \Var (G_z)$. Then
 $$
\P \( \left| \max_{z \in B} G_z - \E \max_{z \in B} G_z \right| \ge a \) \le 2 e^{-\frac {a^2}{2 \s^2}}  \ \ \mbox{ for all } a > 0.
 $$
 \end{lemma}

\section{Hierarchical structure of a path} \label{Sect.paths}

In this section, we will explore a certain hierarchical structure of a path. The whole analysis is purely a geometric and deterministic statement. Later in Section~\ref{Sect.proofs}, a certain multi-scale analysis of the DGFF will be carried out on this hierarchical structure.

In Section~\ref{Sect.sub-paths}, we will consider scales in the form of $K^j$ for all $j \ge 0$. Recall that $K = 2^k$ is large but fixed. We will give the definition of a path in scale $K^j$ (see \eqref{Eq.SLj} below). Then, we will devise a procedure to extract some disjoint sub-paths $P^{(i)}$'s in scale $K^j$ from a path $P$ in scale $K^{j+1}$, which do not necessarily compose a partition of $P$. We call these $P^{(i)}$'s the child-paths of $P$. With this procedure, we can obtain some properties of the child-paths (as stated in the main result Proposition~\ref{Prop.sub-paths} of Section~\ref{Sect.sub-paths}), which are crucial for our goal. Furthermore, we can extract child-paths of $P^{(i)}$'s and so on, and finally obtain a family of sub-paths of $P$. Identifying each sub-path in this family with a node on a tree (where the parent/child relation in the tree is induced by the child-paths under the procedure), we can associate to the path $P$ a tree $\t_P$. With such correspondence, Proposition~\ref{Prop.sub-paths} is translated into Proposition~\ref{Prop.tree}.

In Section~\ref{Sect.Lemma2}, we will prove that $\t_P$ satisfies a certain \emph{regularity} condition if the length exponent of $P$ is close to 1 (see Proposition~\ref{Lem.Lemma2}).

\subsection{Construction of the tree associated to a path} \label{Sect.sub-paths}

We start with a few definitions. Let
 $$
\BD_r = \left\{ \left[ a r  - \frac 1 2 , (a+1) r  - \frac 1 2 \right] \times \left[ b r  - \frac 1 2 , (b+1) r - \frac 1 2 \right] : a, b \in \Z \right\} .
 $$
For technical reasons, we will partition $\R^2$ into boxes whose boundaries can not intersect with $\Z^2$. To this end, we shift the integer boxes by $\frac 1 2$ in each axis. Note that $(B \cap \Z^2)$'s partition $\Z^2$, where $B$ is taken over $\BD_r$.

We will regard a discrete path $P = (z_0, z_1, z_2, \cdots)$ on $\Z^2$ as a continuous path in $\R^2$. Concretely, we identify $P$ as $(z_t, t \ge 0)$ with $z_{n+s} = (1-s) z_n + s z_{n+1}$ for all $n=0,1,2,\ldots$ and $0 \le s \le 1$. To allow us further flexibility, in what follows we allow a path to start or end at a point that lies inside an edge.

Denote by $x_P$ and $y_P$ the starting point and the ending point of a path $P$, respectively. Let $\| P \| : = \| x_P - y_P \|$, similar to discrete paths. Set $|P| : = |P \cap \Z^2|$, which is the cardinality of integer vertices on $P$. We say that two paths $P$ and $Q$ are disjoint if $P \setminus \{ x_P, y_P \}$ and $Q \setminus \{ x_Q, y_Q \}$ are disjoint.

\medskip

\noindent{\bf 1. Paths in scale $K^j$.}

Recall that $B(x, \ell)$ is the $\ell_2$-ball centered at $x$ and of radius $\ell$. Let
 $$
\SL_0 = \Z^2,
 $$
\begin{equation} \label{Eq.SLj}
\SL_j := \left\{ P : \  1 \le \frac 1 {K^j} \|P\| \le 1 + \frac 1 K , \ P \subseteq B \big( x_P, \| P \| \big) \right\} \ \  \mbox{ for all } j \ge 1.
\end{equation}

The property $P \subseteq B(x_P, \| P \|)$ is crucial in the proofs later, and we will show that the extracted child-paths satisfy it (see Proposition~\ref{Prop.sub-paths} (d) below). In the definition of $\SL_0$ and what follows, we identify a vertex $z$ as $\{ z \}$, and regard it as a path in the smallest scale. We say that $P$ is {\em a path in scale $K^j$} if $P \in \SL_j$, in which case $\| P \|$ is comparable to $K^j$ (for $j \ge 1$) .

For $j \ge 0$ and each $P \in \SL_{j+1}$, let
 $$
E (P) := \left\{ z \in \R^2 : \| x_P -z \| + \| y_P -z \| \le \( 1 + \frac 2 {K^2} \) \| P \|  \right\} ,
 $$
which is an ellipse of width $O(K^{j+1})$ and height $O(K^j)$. Define
 \begin{equation} \label{Eq.tildeExy}
\tilde E (P) = \left\{ z \in \R^2 : d \big( z, E(P) \big) \le 4 K^j  \right\} .
 \end{equation}
We say that $P$ is \emph{tame} if $P \subseteq \tilde E(P)$, and {\em untamed} otherwise. We would like to mention that we do not say a path in $\SL_0$ is tame or untamed.

Rescale $\R^2$ by regarding $K^j$ as unity, and consider the paths in scale $K^{j+1}$. The tame property implies that $P$ is roughly a straight segment of length $O(K)$, while the untamed property implies that $P$ looks like a curve. What is important to us is that a curve is longer than a straight segment with the same endpoints. Thus, we can extract more/longer child-paths from untamed path, as Property (b2) in Proposition \ref{Prop.sub-paths} states. In other words, an untamed sub-path will bring a large amount of offsprings in the whole family. Therefore, if a path in $\p_\k$ has length exponent close to 1, the family of its sub-paths can not contain many untamed ones, which will be proved in Proposition~\ref{Lem.Lemma2}.

We call Property (c) in Proposition \ref{Prop.sub-paths} the 12-times-rule (here the number 12 is a somewhat arbitrary choice). With it, the child-paths are not dense in some sense, so that we can take advantage of \eqref{Eq.Xindpt} and employ a large deviation estimate for $h_j |_{P^{(i)}}$'s later. Recall $[r] = \{ 1, \ldots, r \}$.

 \begin{prop} \label{Prop.sub-paths}
Suppose that $j \in [m-2]$ and $P \in \SL_{j+1}$. Then, there exists $\ell \in [K^j, (1 + \frac 1 K ) K^j]$, a positive integer $d$, and disjoint sub-paths $P^{(i)}$ of $P$ for $i \in [d]$ such that the following hold.

(a) $d \ge K$.

(b1) $\ell \ge \frac 1 d \| P \|$.

(b2) $\ell \ge (1 + \frac 1 {K^2})  \frac 1 d  \| P \|$ if $P$ is untamed.

(c) Each box in $\BD_{K^j}$ is visited by at most 12 sub-paths of the form $P^{(i)}$, $i \in [d]$.

(d) $P^{(i)} \in \SL_j$, and $\| P^{(i)} \| = \ell$ for each $i\in [d]$.
 \end{prop}
 \begin{proof}
In the following, let $x, y$ be respectively the starting and ending points of $P$, and let $x_i, y_i$ be those of $P^{(i)}$. For $z\in P$, let $P_z$ be the sub-path of $P$ starting at $z$, and let $P_{z, \ell}$ be the sub-path of $P$ starting at $z$ until $P$ reaches $\partial B(z, \ell)$ for the first time. We will set $\ell$, pick $x_i$'s, and define $y_i$ as the point where $P_{x_i}$ first reaches $\partial B(x_i, \ell)$, i.e. $P^{(i)} = P_{x_i, \ell}$. Then, Property (d) will be satisfied naturally, and only Properties (a), (b1), (b2) and (c) are to be checked. Next, we will describe the procedure and verify our construction respectively in the tame and untamed cases.

\noindent {\bf Suppose that $P$ is tame.} Set $d : = K$ and $\ell : = \frac 1 K \| P \|$. Let $L(x,y) $ be the straight segment connecting $x$ and $y$, i.e., $L(x,y) = \{ (1-s) x + s y: 0 \le s \le 1 \}$. Denote by $z_i$ the point in $L(x,y)$ with $\| z_i - x \| = i \ell $, and by $L_i$ the line perpendicular to $L(x,y)$ at $z_i$.  Then we take $x_i$ to be the point where $P$ hits $L_{i-1}$ for the first time, for $i \in [d]$, with $x_1 = x$. It is not hard to check that $P^{(i)}$'s are disjoint, and Properties (a), (b1), (c) all hold.

 \begin{figure}[h]
  \includegraphics[width=17cm]{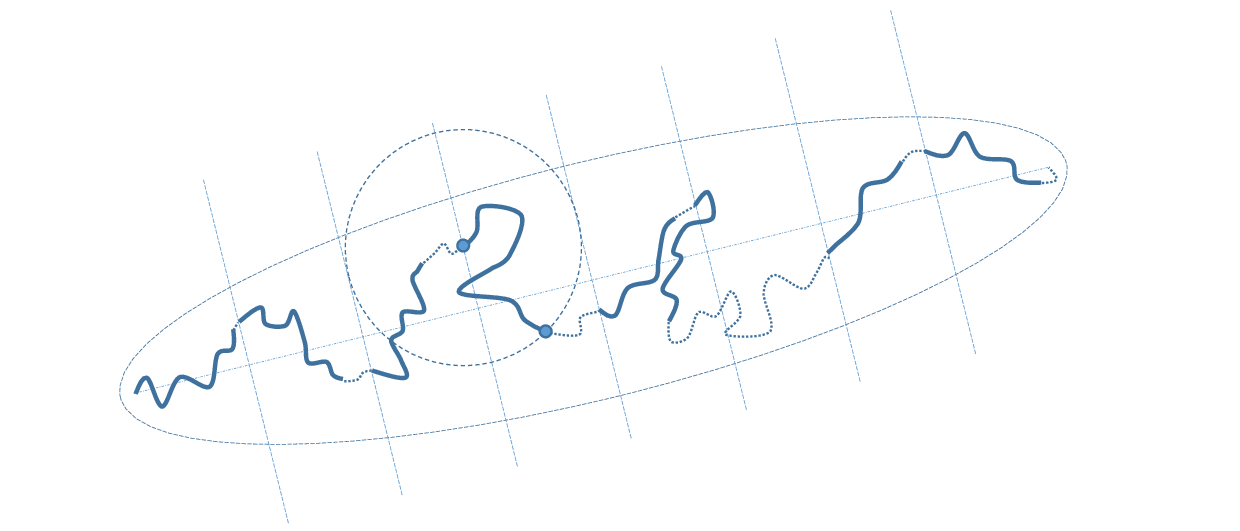}
\\ \vspace{-1cm} \caption{The tame case. $K=8$. The solid lines are $P^{(i)}$'s. The dotted lines will be deleted. The circle along the dashed line has radius $\ell$, centered at $x_4$.} \end{figure}

\noindent {\bf Suppose that $P$ is untamed.}  Set
 $$
\hat d : = \left\{ \begin{array}{ll} K, & \mbox{if }   \| P \| \le \frac {1 + 1 / K}{1 +  1 / K^2 } K^{j+1} ,  \\ K+1, & \mbox{otherwise,} \end{array} \right.  \ \ \ \mbox{and} \ \ \
\ell : = \frac 1 {\hat d} \( 1 + \frac 1 {K^2} \) \| P \|.
 $$
Since $P \in \SL_{j+1}$, we always have $\ell \in \left[ K^j, K^j \big( 1 + \frac 1 K \big) \right] $.

We set $x_1 = x$. The key is to describe the choice of $x_{i+1}$ for $i\geq 1$ recursively. Suppose that we have defined $x_i$ and $y_i$. To obtain Property (c),  let $\BD^i$ be the collection of boxes $B \in \BD_{K^j}$ such that $B$ has been visited by 11 sub-paths in $P^{(1)}, \cdots, P^{(i-1)}$, and it is also visited by $P^{(i)}$ (note that $\BD^i$ could be empty). If $\BD^i = \emptyset$, we define $x_{i+1} = y_i$. If $\BD^i \neq \emptyset$ and $y \notin B^i : =  \cup_{B \in \BD^i} B$, let $z_i$ be the point where $P$ exits $B^i$ eventually, i.e.
 $$
z_i \in B^i \ \ \ \mbox{and} \  \ \ P_{z_i} \setminus \{ z_i \} \subseteq (B^i)^c.
 $$
In fact, $z_i \in \partial B^i$ since $P$ is a continuous path in $\R^2$. If $y_i$ appears after $z_i$ along $P$, i.e. $z_i \in P^{(i)}$, we define $x_{i+1} = y_i$. Otherwise, we define $x_{i+1} = z_i$, thus the part in $P$ from $y_i$ to $z_i$ (except the endpoints) will be removed and not be contained in any sub-path. In words, we define $x_{i+1}$ as the later one of $y_i$ and $z_i$ along $P$. This will imply Property (c) and that sub-paths are disjoint. By Property (d), $P^{(i)} \subseteq B(y_i, 2 \ell)$. Note that the diameter of a box in $\BD_{K^j}$ is $\sqrt 2 K^j \le 2 \ell$, and $x_{i+1} = y_i$ or $z_i \in B^i$. By the definition of $\BD^i$,
 \begin{equation} \label{Eq.distofxyz}
\| w - y_i \| \le 4 \ell \ \mbox{ for all }w \in B^i, \ \ \ \mbox{and} \ \ \ \| x_{i+1} - y_i \| \le 4 \ell.
 \end{equation}
If $\BD^i \neq \emptyset$ and $y \in B^i$, we do not define $x_{i+1}$ and stop the procedure. To summarize, we have
 \begin{equation} \label{Eq.xi3case}
\left\{ \begin{array}{ll}  x_{i+1} : = y_i & \mbox{in Case 1, } \BD^i  = \emptyset, \\ x_{i+1} : = \mbox{ the later one of } y_i \mbox{ and } z_i & \mbox{in Case 2, } \BD^i \neq \emptyset \mbox{ and } y \notin B^i, \\ x_{i+1} \mbox{ is not defined} & \mbox{in Case 3, } y \in B^i.  \end{array} \right.
 \end{equation}
In Case 1 and 2, $y_{i+1}$ can not be defined if $P_{x_{i+1}}$ is contained in the interior of $B(x_{i+1}, \ell)$. Let
 $$
d : = \min \{ i : x_{i+1}  \mbox{ or } y_{i+1} \mbox{ is not well-defined} \}.
 $$
That is, the procedure stops at $i = d$ naturally, when it can not continue. By our construction, the sub-paths are disjoint, and any box $B \in \BD_{K^j}$ can be visited by at most 12 sub-paths, verifying Property (c). Thus, it remains to check Properties (a) and (b2).

 \begin{figure}[h]
  \includegraphics[width=17cm]{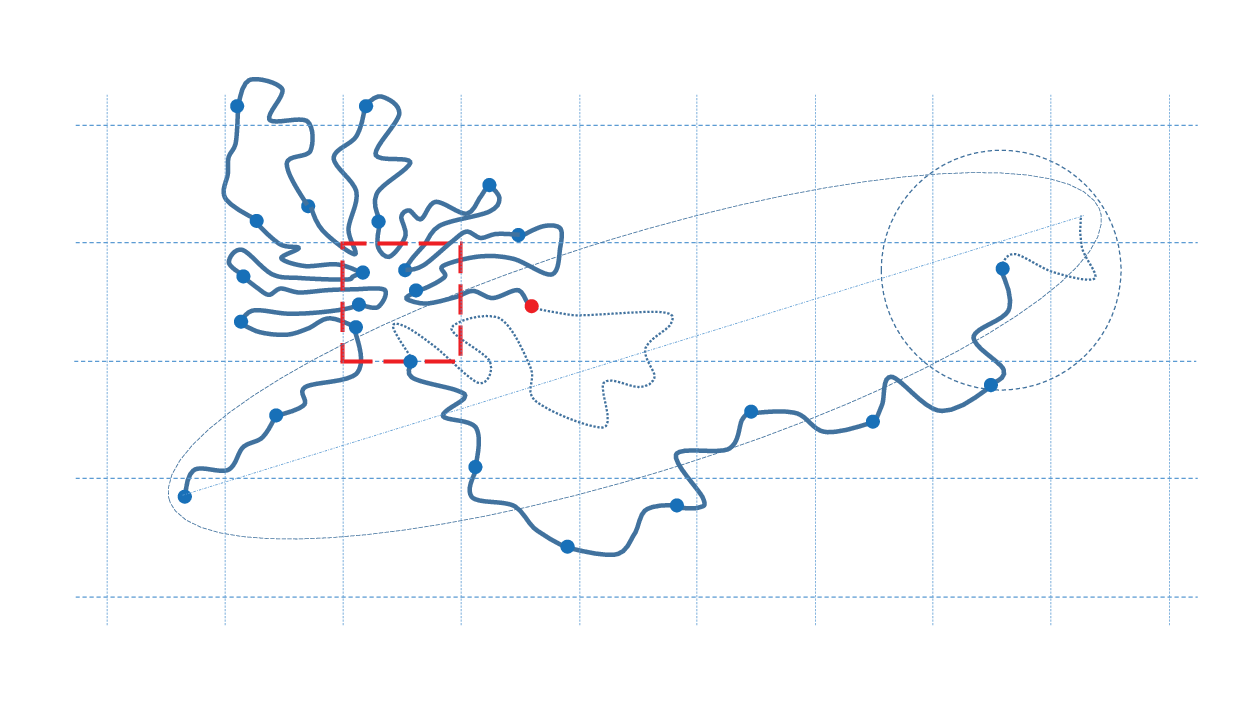}
\\ \vspace{-1.5cm} \caption{The untamed case. $K=8$. The solid lines are sub-paths, with endpoints being marked with round dots. The 12-times-rule is applied to the red box $B \in \BD_{K^j}$, so that the dotted line from the red round dot will be deleted. } \end{figure}

Next, we will prove $d \ge \hat d$ in three steps. Assuming this, we will have Property (a) and Property (b2) by the choice of $\hat d$ and $\ell$, completing the proof.

\medskip

\noindent {\bf Step 1}.  We will show that the sequence of sub-paths spreads out regularly in the sense that
 \begin{equation} \label{Eq.sub-pathcont}
\| z - x \| \le i \ell \ \mbox{ for all } z \in P^{(i)} \mbox{ and } i \in [d].
 \end{equation}
We will check it by induction on $i$. Note that \eqref{Eq.sub-pathcont} holds for $i=1$ by the definition of $P^{(1)}$. Suppose that it holds for all $i^\prime \in [i]$, then we need to show that it also holds for $i+1 (\le d)$. By Property (d) and the triangle inequality, we only need to check that $\| x_{i+1} - x \| \le i \ell$. If $x_{i+1} = y_i$, it holds by the induction hypothesis and $y_i \in P^{(i)}$. Otherwise, Case 2 in \eqref{Eq.xi3case} is true, and $x_{i+1} = z_i$, which lies in some box $B \in \BD^i$. By the definition of $\BD^i$, one can find $i^\prime \le i-11$ and a point $w \in B \cap P^{(i^\prime)}$ such that $\| x_{i+1} - w \| \le \sqrt 2 K^j \le 2 \ell$. Combined with the induction hypothesis, this yields that $\| x_{i+1} - x \| \le \| x_{i+1} - w \| + \| w - x \| \le  2 \ell + i^\prime \ell \le i \ell$.

\medskip

\noindent {\bf Step 2}. We will show that
 \begin{equation} \label{Eq.distydy}
\| y_d - y \| \le 5 \ell .
 \end{equation}

 \begin{itemize}

\item Suppose that $x_{d+1}$ is well-defined. Then, $\| x_{d+1} - y_d \| \le 4 \ell$ by \eqref{Eq.distofxyz} and \eqref{Eq.xi3case}. In this case, $\| x_{d+1} - y \| < \ell$ since $y_{d+1}$ can not be well-defined. By the triangle inequality, $\| y_d - y \| \le 4 \ell + \ell = 5 \ell$.

\item Suppose that $x_{d+1}$ is not well-defined. Then, Case 3 in \eqref{Eq.xi3case} is true, i.e. $y \in B^d$. Then, \eqref{Eq.distydy} follows from \eqref{Eq.distofxyz}.

\end{itemize}

\medskip

\noindent {\bf Step 3}. We will show $d \ge \hat d$, using the untamed property and results in the previous two steps, respectively in the scenarios whether the 12-times-rule has been invoked or not.

{\bf Suppose that some box in $\BD_{K^j}$ is visited by 12 sub-paths.} Then, we claim that
 \begin{equation} \label{Eq.distofxy}
\| x-y \| \le (d-2) \ell.
 \end{equation}
Assuming \eqref{Eq.distofxy}, we will conclude that $d \ge 2 + \frac {\| x-y \|}{\ell} \ge K + 1 \ge \hat d$, where we have used $\frac {\| x-y \|} \ell = \frac {\hat d }{ 1 + 1 / K^2 } \ge \frac K {1 + 1 / K^2 } > K - 1$ by the choices of $\hat d$ and $\ell$. Next, we verify \eqref{Eq.distofxy}. Let $i$ satisfy $B^i \neq \emptyset$ and $B^{i^\prime} = \emptyset$ for all $i^\prime > i$. We pick a box $B \in \BD^i$, and let $P^{(i_0)}$ be the first sub-path visiting $B$. Then, $i_0 \le i - 11$. Pick $w \in B \cap P^{(i_0)}$, we have $\| x-w \| \le i_0 \ell \le (i - 11) \ell $ by \eqref{Eq.sub-pathcont}.
 \begin{itemize}

\item Suppose $i=d$. Combining the preceding display with \eqref{Eq.distydy} and \eqref{Eq.distofxyz}, we have $\| x-y \| \le \| x - w \| + \| w-y_d \| + \| y_d - y\| \le (d-2) \ell$.

\item Suppose $i < d$. Let $B$ be the box containing $z_i$. Then, $\| w - x_{i+1} \| \le \max \left\{ 4 \ell, \sqrt 2 K^j \right\} = 4 \ell$, where we have used \eqref{Eq.distofxyz} and the fact that $x_{i+1} = y_i$ or $z_i$. By the choice of $i$, $\| x_{i+1} - y_d \| \le (d - i) \ell$. Combining the above inequalities, we have that
$$\| x-y \| \le \| x-w \| + \| w - x_{i+1}\| + \| x_{i+1} - y_d \| + \| y_d - y \| \le (i-11) \ell + 4 \ell + (d-i) \ell + 5 \ell = (d-2) \ell\,,$$ where we have used \eqref{Eq.distydy}.
 \end{itemize}

{\bf Suppose that none of the boxes in $\BD_{K^j}$ is visited by 12 sub-paths.} Then, we have not removed any part of $P$ in the procedure. By the definition of $d$, we conclude $\| y - y_d \| < \ell$ and $P_{y_d} \subseteq B (y, 2 \ell)$. It follows that every vertex in $P \setminus B(y, 2 \ell)$ must lie in some sub-path. Note that $B(y, 2 \ell) \subseteq \tilde E(P)$ for $2 \ell \le 4 K^j$. Since $P$ is untamed, there exist $w \notin \tilde E (P)$ and $i \le d-1$ such that $w \in P^{(i+1)}$. By Property (d) and $\ell < 4 K^j$, we have that $x_{i+1} \notin E (P)$, i.e., $\| x - z \| + \| z - y \| \ge (1 + \frac 2 {K^2}) \| P \| > \hat d \ell$, where $z = x_{i+1} = y_i$. Since $x_{i^\prime+1} = y_{i^\prime}$ for all $i^\prime$, we have $i \ell \ge \| x- y_i \| = \| x-z \|$ and $(d-i) \ell \ge \| x_{i+1} - y_d \| = \| z - y_d \|$ by Property (d) and the triangle inequality. Combining the inequalities above, we conclude that $ (d+1) \ell = i \ell + (d-i) \ell + \ell \ge \| x-z \| + \| z- y_d \| + \| y_d - y \| \ge \| x-z \| + \| z - y\| > \hat d \ell$. Therefore, $d+1 > \hat d$, i.e. $d \ge \hat d$.
 \end{proof}

For $j = 0$ and $P \in \SL_{j+1} = \SL_1$, we regard each $\{ z \}$ as a sub-path of $P$, where $z \in (\Z^2 \cap P) \setminus \{ x_P , y_P  \}$. We remove the endpoints to match the definition that paths are disjoint. Note $K \le \| P \| \le K + 1$ for $P \in \SL_1$. It follows that $|P| \ge K - 2$. Consequently, $P$ has at least $K-4 \ge \frac 1 2 \| P \|$ sub-paths.
\begin{cor}\label{Cor.SLm}
Suppose $j = 0$ and $P \in \SL_1$. Then, we can find $d \ge \frac 1 2 \| P \|$ different sub-paths $P^{(i)}$, $i \in [d]$ in the interior of $P$.
 \end{cor}

\medskip

\noindent{\bf 2. \bf Constructing the tree $\t_P$.}

For each $j \ge 0$, we will associate every path $P \in \SL_j$ with a tree $\t_P$ of depth $j$. We will first introduce the construction of $\t_P$, then state its properties in Proposition~\ref{Prop.tree} below (which is merely a translation of Proposition~\ref{Prop.sub-paths} and Corollary~\ref{Cor.SLm} in the language of trees). In what follows, we denote by $d_P$ the number $d$ in Proposition~\ref{Prop.sub-paths} and Corollary~\ref{Cor.SLm} for clearer dependence on the path $P$. Recall that we call the sub-paths $P^{(i)}$, $i \in [d]$ in Proposition~\ref{Prop.sub-paths} and Corollary~\ref{Cor.SLm} child-paths of $P$.

We will construct the trees by induction on $j$. For $j = 1$, we identify $P$ as the root $\rho$, and each child-path constructed in Corollary~\ref{Cor.SLm} as a child of $\rho$. In this case, each child is a leaf and is identified as a vertex in $\Z^2 \cap \big( P \setminus \{ x_P, y_P \} \big)$. For $j \ge 2$, suppose that $P \in \SL_j$, and that we have associated every path $Q \in \SL_{j-1}$ with a tree $\t_Q$ of depth $j-1$. By Proposition~\ref{Prop.sub-paths}, we extract $d = d_P$ disjoint child-paths $P^{(1)}, \cdots, P^{(d)} \in \SL_{j-1}$ from $P$. Then, we associate the root $\rho$ with $d$ children $u_1, \cdots, u_d$. In addition, we identify $P$ as $\rho$, and identify $P^{(i)}$ as $u_i$ for all $i \in [d]$. By attaching the root of $\t_{P^{(i)}}$ to the node $u_i$, we obtain the tree $\t_P$.

So far, we have associated a path $P \in \SL_j$ with a tree $\t_P$ of depth $j$. Denote the root by $\rho$, and the set of leaves by $\l$. Each node $u$ is identified with a sub-path of $P$, which is denoted by $P^u$. Let $L(u)$ be the level of a node $u$, with $L(\rho) = 0$ and $L(v) = j$ for $v\in \l$ (thus we have $P^u \in \SL_{j - L(u)})$. Let $\| u \| : = \| P^u \|$ if $L(u) \le j-1$. Let $d_u = d_{P^u}$, which is the number of children of $u$. We say that $u$ is tame/untamed if so is $P^u$. We would like to mention that different leaves correspond to different vertices, since child-paths are disjoint and we have removed endpoints of paths in $\SL_1$ in constructing leaves (see Corollary~\ref{Cor.SLm}).

 \begin{prop} \label{Prop.tree}
Suppose that $P \in \SL_j$, $j \in [m-1]$. Then, the tree $\t_P$ associated with $P$ constructed above satisfies the following properties.

(a1) $d_u \ge K$ if $L(u) \le j-2$.

(a2) $d_u \ge \frac 1 2 \| u \|$ if $L(u) = j - 1$.

(b1) $\| v \| \ge \frac 1 {d_u} \| u \|$ if $L(u) \le j-2$, where $v$ is a child of $u$.

(b2) Furthermore, $\| v \| \ge  (1 + \frac 1 {K^2}) \frac 1 {d_u} \| u \|$ if $u$ is untamed.

(c) $|\l| \le |P|$.
 \end{prop}

\medskip

\noindent{\bf 3. Paths in $\p_k$.}

We will investigate paths in $\p_\k$, where we recall $\p_\k = \{ P : \| P \| \ge \k N \}$. Let
 \begin{equation} \label{Eq.d0}
d_0 :=  \lfloor \frac {\k N}{K^{m-1}} \rfloor \ge K,
 \end{equation}
where we recall that $\k N \ge K^{m+1}$, and $\lfloor a \rfloor$ is the largest integer which is less than or equal to $a$. We follow the procedure to extract child-paths from a tame path in the proof of Proposition~\ref{Prop.sub-paths}, with $\ell : = K^{m-1}$ in the scenario here. Consequently, we can extract $d_0$ sub-paths $P^{(i)} \in \SL_{m-1}$, $i \in [d_0]$ from $P$, with $P_{y_{d_0}}$ being removed completely. Then, we obtain Corollary~\ref{Cor.Pkpath} below. Furthermore, let the root $\rho$ have $d_0$ children and attach the roots of $\t_{P^{(i)}}$'s to them, and then we obtain the tree $\t_P$. Following Proposition~\ref{Prop.tree}, we have Corollary~\ref{Cor.Pktree} below.

 \begin{cor} \label{Cor.Pkpath}
Suppose $P \in \p_\k$. We can extract $d_0$ disjoint sub-paths (of $P$) in $\SL_{m-1}$ such that (c) in Proposition~\ref{Prop.sub-paths} holds (with $j = m-1$).
\end{cor}

 \begin{cor} \label{Cor.Pktree}
Suppose $P \in \p_\k$. Set $j=m$. Then, $d_\rho = d_0$, (a1), (a2), (c) in Proposition~\ref{Prop.tree} hold, and (b1), (b2) in Proposition~\ref{Prop.tree} hold for $u \neq \rho$.
 \end{cor}

\subsection{The total flow through untamed nodes}  \label{Sect.Lemma2}

Recall that every $P \in \p_\k$ is associated with a tree $\t_P$ of depth $m$. Let $\theta_P$ be {\em the unit uniform flow} on $\t_P$ from $\rho$ to $\l$, with $\theta_P (\rho) = 1$ and $\theta_P (v) = \frac 1 {d_u} \theta_P (u)$ if $v$ is a child of $u$. Let
 $$
\p_{\k, \d, K} : = \left\{ P :  P \mbox{ is a path in } V_N, \ \| P \| \ge \k N \mbox{ and } |P| \le N^{1 + \frac \d {K^2 k}} \right\} .
 $$
In this section, we will show the following proposition.
 \begin{prop} \label{Lem.Lemma2}
For each $P \in \p_{\k, \d, K}$,
 $$
\sum_{u: 1 \le L(u) \le m-1} \theta_P (u) 1_{\{ u \mbox{ is untamed} \} } \le 2 \d m .
 $$
 \end{prop}
 \begin{proof}
Note $\p_{k, \d, K} \subseteq \p_\k$. For every leaf $v$, denote nodes on the ray in $\t_P$ from $\rho$ to $v$ by $v_0 (= \rho), v_1, \cdots, v_{m-1, }v_m (= v)$ in order. Let $C_v =  |\{ 1 \le r \le m-1: v_r \mbox{ is untamed} \} |$. Then, by Corollary~\ref{Cor.Pktree} and (a2), (b1), (b2) in Proposition~\ref{Prop.tree},
 $$
\prod_{r=1}^{m-2} \frac { \| v_{r+1} \| }{\| v_r \|} \ge  \( 1 + \frac 1 {K^2} \) ^{C_v - 1} \times \prod_{r=1}^{m-2} \frac 1 {d_{v_r} } \ \ \ \ \ \mbox{ and } \ \ \ \ \ \frac 1 {\| v_{m-1} \|} \ge \frac 1 {2 d_{v_{m-1}} } .
 $$
Note that $\| v_1 \| \ge K^{m-1} $, since $P^{v_1} \in \SL_{m-1}$. It follows that
 $$
\frac 1 {K^{m-1}} \ge \frac 1  {\| v_1 \|} = \( \prod_{r=1}^{m-2} \frac {\| v_{r+1} \|}{\| v_r \|} \) \frac {1}{\| v_{m-1} \|} \ge  \frac 1 2 \( 1 + \frac 1 {K^2} \) ^{C_v - 1}  \times \prod_{r=1}^{m-1} \frac 1 {d_{v_r} }  .
 $$
Note that $d_{v_0} \ge K$ by \eqref{Eq.d0}. By definition, $\theta_P (v) \le \frac 1 K \prod_{r=1}^{m-1} \frac 1 {d_{v_r} }$. Consequently,
 \begin{equation} \label{Eq.thetav}
\theta_P (v) \le  \frac 2 {K^m}  \( \frac {K^2} {K^2 + 1} \) ^{C_v - 1} \le \frac 4 {K^m} e^{ - \frac 1 {K^2 + 1} C_v }.
 \end{equation}

It follows that
 \begin{eqnarray*}
  & &
\sum_{u: 1 \le L(u) \le m-1} \theta_P (u) 1_{ \{ u \mbox{ is untamed} \} }
 \\ & = &
\sum_{u: 1 \le L(u) \le m-1} \( \sum_{v \in \l}  \theta_P (v) 1_{ \{ u \mbox{ is an ancestor of } v \} } \) 1_{ \{ u \mbox{ is untamed} \} }
 \\ & = &
\sum_{v \in \l } \theta_P (v) C_v = \sum_{v \in \l} \theta_P (v) C_v 1_{\{C_v \ge \d m \}} + \sum_{v \in \l } \theta_P (v) C_v 1_{\{ C_v < \d m \}}
 \\ & \le &
\frac 4 {K^m}  e^{ - \frac 1 {K^2 + 1} \d m }   |\l |  m    + \d m ,
 \end{eqnarray*}
where in the last inequality, we have used $C_v \le m$ and \eqref{Eq.thetav} in the case of $C_v \ge \d m$, and have used $\sum_{v \in \l } \theta_P (v) = 1$ in the case of $C_v < \d m$.

Finally, we will show that $\frac 4 {K^m}  e^{ - \frac 1 {K^2 + 1} \d m }   |\l | \le \d$, which will then complete the proof of the proposition. By (c) in Proposition~\ref{Prop.tree} and the assumption $P \in \p_{\k, \d, K}$, and recalling $\k N < K^{m+2}$, we have that $| \l | \le |P| \le N^{1 + \frac {\d}{ K^2 k}} \le \( \frac {K^{m+2}}{\k} \) ^{ 1 + \frac {\d} {K^2 k}}$. Therefore,
 $$
\frac 4 {K^m} e^{ - \frac 1 {K^2 + 1} \d m  }  |\l | \le 4 \( \frac {K^2}{\k} \) ^{ 1 + \frac {\d} {K^2 k}}  \( K^{ \frac  1 {K^2 k} } e^{- \frac 1 {K^2 + 1}}  \) ^{\d m}  \le \d ,
 $$
since $ K^{\frac 1 {K^2 k}} e^{ - \frac 1 {K^2+1} } < 1$ and $m \to \infty$.
 \end{proof}

\section{Multi-scale analysis on the hierarchical structure of the path} \label{Sect.proofs}

In this section, we will prove Theorem~\ref{mainthm}. To this end, we say that a path $P$ in $\SL_j$ is {\em open} if $\eta_j (z) \ge \e k$ for some $z \in P \cap \Z^2$, where $\e > 0$ is to be selected and $\eta_j$ is defined in \eqref{Eq.defetaj} above. We say that a node $u$ is open if so is $P^u$. We will show that with high probability, each tame node has a small fraction of open children, which is Proposition~\ref{Prop.Lemma1}. Combing this with Proposition~\ref{Lem.Lemma2}, we will show in Proposition~\ref{Lem.Lemma3} that open nodes in the tree are rare. In Section~\ref{Sect.Pfmain}, we will show that many leaves have a small number of open ancestors. For such a leaf $z$, $\eta_j (z) \le \e k$ for most $j$. Then, we will show that $\eta^{V_{3N}} (z) \le 15 \sqrt \d \log N$, by setting $\e = \frac 1 2 \sqrt{\d}$. As a result, Theorem~\ref{mainthm} will follow by symmetry.

\subsection{The fraction of open child-paths of a tame path} \label{Sect.Lemma1}

Let us begin with some definitions. Suppose that $P \in \SL_j$. Let  $P^{(i)}$, $i \in [d_P]$ be the child-paths constructed in Proposition~\ref{Prop.sub-paths} or Corollary~\ref{Cor.SLm} (depending on the value of $j$). Let
 $$
\D_P = \frac 1 {d_P} \left| \left\{ i \in [d_P] : P^{(i)} \mbox{ is open} \right\} \right| .
 $$
We would like to mention that for each $P \in \SL_j$, $\D_P$ relies on the field $\eta_{j-1}$ since $P^{(i)} \in \SL_{j-1}$ for all $i \in [d_P]$. In what follows, we will often deal with tame paths $P$'s with fixed endpoints, so that $\D_P$'s rely on the Gaussian field in a local region. Concretely, we will regard each box in $\BD_{K^{j-2}}$ as an end-box of paths in $\SL_j$ so that we can classify $\SL_j$ via end-boxes. The main goal of this section is to show that $\D_P$'s for tame paths $P$'s with endpoints lying in a fixed pair of end-boxes are uniformly small (see Proposition~\ref{Prop.Lemma1}). Later, we will use \eqref{Eq.Xindpt} and employ a large deviation estimate to prove Proposition~\ref{Lem.Lemma3}, via counting possible pairs of end-boxes. Define
 $$
\p_j (B_1, B_2) := \big\{ P \in \SL_j : x_P  \in B_1  \mbox{ and }   y_P \in B_2 \big\} , \ \ T_j (B_1, B_2) : = \big\{ P \in \p_j (B_1, B_2) : P \mbox{ is tame} \big\} ,
 $$
 $$
\mbox{and } \ \END_j : = \big\{ (B_1, B_2): B_i \in \BD_{K^{j-2}} \mbox{ and } B_i \cap V_N \neq \emptyset \mbox{ for } i = 1,2, \mbox{ and } \p_j (B_1, B_2) \neq \emptyset \big\} .
 $$

 \begin{prop} \label{Prop.Lemma1}
There exist universal constants $C_7, C_8 > 0$ and a large $K_1 (\e) > 0$ such that the following holds for all $K \ge K_1 (\e)$. For each $j \in [m-1]$ and each $(B_1, B_2) \in \END_j$,
 $$
\P \(  \D_P \ge \d \ \mbox{ for some } P \in T_j (B_1, B_2) \) \le K^2 e^{- C_7 \e^2 k^2 \d^2 } \ \ \ \mbox{ for all }  \d \ge  \frac { C_8 }{\sqrt k \e}.
 $$
 \end{prop}

To this goal, we rescale $\R^2$ by regarding $K^{j-1}$ as unity. Then, roughly speaking, under the new scale,  paths in $T_j (B_1, B_2)$ all lie in the same ellipse with width $O(K)$ and height $O(1)$. This ellipse is what the set $D$ below stands for, where \eqref{Eq.ellipseshape} below describes the width and the height. We will show in Lemma~\ref{Lem.Lemma1} that a stronger result holds in the new scale of $\R^2$, and then prove Proposition~\ref{Prop.Lemma1}.

Suppose $D \subseteq \Z^2$. Let
 $$
D_{1,a} : = \left| \big\{ b \in \Z : (a, b) \in D \big\} \right| , \ \ \ \ \ D_1 = \max_{a \in \Z} D_{1,a},
 $$
 $$
D_{2,b} =  \left| \big\{  a \in \Z : (a, b) \in D \big\} \right| , \ \ \ \ \ D_2 = \max_{b \in \Z} D_{2, b}.
 $$

 \begin{lemma} \label{Lem.Lemma1}
Let $c_1, c_2, c_3 > 0$. Suppose that $D \subseteq \Z^2$ satisfies
 \begin{equation} \label{Eq.ellipseshape}
|D| \le c_1 K \ \ \ \mbox{ and } \ \ \ D_1 \wedge D_2 \le c_1.
 \end{equation}
Suppose that $\varphi_r$, $0 \le r \le k -1$ are mean-zero Gaussian fields in $D$ satisfying the following (i)-(iii).
\begin{enumerate}[ (i) ]
\item For each $z \in D$, $\varphi_0 (z), \cdots, \varphi_{k-1} (z)$ are independent.  For each $r$, $\varphi_r (z)$ is independent of $\varphi_r (w) $ if $|z-w|_\infty \ge 2^{r+1}$.

\item  $\E \varphi_r (z)^2 \le c_2$ for all $r, z$.

\item $\E \varphi (z) \varphi (w) \ge - c_3$ for all $z, w \in D$, where $\varphi :=  \sum_{r=0}^{k-1} \varphi_r$.

\end{enumerate}
Then, for $k \ge \frac {c_3}{c_2} \vee 6$,
 $$
\P \( \left| \big\{ z \in D : \varphi(z) \ge \e k \big\} \right| \ge \d K \) \le 5 k e^{- c \e^2 k^2 \d^2} \ \ \ \mbox{ for all } \  \frac {8 c_1 \sqrt {2 c_2}}{\sqrt k \e } \le \d \le 1 ,
 $$
where $c = \min \left\{ \frac 1 {128 c_3}, \frac 1 {128 c_1^2 c_3 },  \frac 1 {512 c_1^2 c_2} \right\} $.
 \end{lemma}
 \begin{proof}
To investigate the event $\big| \{ z \in D : \varphi(z) \ge \e k \} \big| \ge \d K$, we will compare the value of $\sum_{z \in D} \varphi (z)$ with $\e k \d K$. To this goal, we will add a common variable to the field $\varphi$ in order to obtain a positively correlated Gaussian field and employ the FKG inequality. Let $Z \sim N(0, c_3)$ be independent of the Gaussian field $\varphi$. Let
 $$
E_\d = \left\{ \left| \left\{ z \in  D : \varphi (z) + Z \ge \frac 7 8 \e k \right\} \right| \ge \d K \right\} .
  $$
Then,
  \begin{equation} \label{Eq.decomde}
\P \( \left| \big\{ z \in D : \varphi(z) \ge \e k \big\} \right| \ge \d K \) \le \P \( - Z \ge \frac 1 8 \e k \) + \P ( E_\d ) \le 2 e^{ -\frac 1 {128 c_3} \e^2 k^2 \d^2} + \P ( E_\d ) ,
  \end{equation}
where in the second inequality we have used Lemma~\ref{Lem.concentration} and $\d \le 1$. Next, we estimate $\P (E_\d)$. Define
   $$
E = \left\{ \sum_{z \in D} \big| \varphi (z) + Z \big| 1_{\{ \varphi (z) + Z  < 0 \}} \le 2  c_1 \sqrt {2 c_2 k} K  \right\} .
 $$
By independence, $\E \big( \varphi (z) + Z \big) ^2 = \E \varphi (z)^2 + \E Z^2 \le c_2 k + c_3 \le 2 c_2 k$, where we have used (i), (ii) as well as the assumption $k \ge c_3 / c_2$. Consequently, $\E | \varphi (z) + Z| \le \sqrt {2 c_2 k}$. It follows that
 $$
\P (E^c) \le \P \( \sum_{z \in D} |\varphi (z) + Z |  >  2  c_1 \sqrt {2 c_2 k} K \)  \le \frac 1 { 2  c_1 \sqrt {2 c_2 k} K }  |D| \sqrt {2 c_2 k} \le \frac 1 2 ,
 $$
where we have used \eqref{Eq.ellipseshape} in the last inequality. That is, $\P (E) \ge \frac 1 2$. Note that $\{ \varphi (z) + Z, z \in D \}$ is positively correlated by (iii) as well as the choice of $Z$, and that both $E_\d$ and $E$ are increasing events of it. By the FKG inequality  \cite{Pitt82}, $\P (E_\d) \P (E) \le \P (E_\d \cap E)$. Thus,
  $$
\P(E_\d)
 \le
\frac {\P(E_\d \cap E) }{\P(E)} \le 2 \P (E_\d \cap E) \le 2 \P \( \sum_{z \in D} ( \varphi (z) + Z) \ge \frac 7 8 \e k  \d K - 2 c_1  \sqrt {2 c_2 k} K \)  .
 $$
Note that $\frac 7 8 \e k  \d K - 2 c_1  \sqrt {2 c_2 k} K \ge \frac 5 8 \e k  \d K  $ since $\d \ge 8c_1 \sqrt {2 c_2} / (\sqrt k \e)$ by the assumption. Then, it follows that
  \begin{eqnarray}
\P(E_\d)
 & \le &
2 \P \(  |D| Z \ge \frac 1 8 \e \d K k  \)  + 2 \P \( \sum_{z \in D} \varphi (z) \ge  \frac 1 2  \e \d K k \) \nonumber
 \\ & \le &
4 e^{- \frac 1 {2 c_3} \cdot \frac 1 {64 c_1^2 } \e^2 k^2 \d^2 } +  2 \P \( \sum_{z \in D} \varphi (z) \ge   \frac 1 2  \e \d K k  \) , \label{Eq.tildeEdEd}
\end{eqnarray}
where in the second inequality, we have used Lemma~\ref{Lem.concentration} and \eqref{Eq.ellipseshape}.

Next, we estimate $\P \( \sum_{z \in D} \varphi (z) \ge   \frac 1 2  \e \d K k  \) $. Let
 $$
\Gamma : =  \frac 1 2 \e \d K k \ \ \ \mbox{ and } \ \ \ \Gamma_r := 2^{- \frac 1 2 (k-r) - 3} \e \d K k .
 $$
Note $\Gamma_r \le \( 1- \frac 1 {\sqrt 2} \) 2^{- \frac 1 2 (k-r)} \Gamma$. So, $\sum_{r=0}^{k-1} \Gamma_r \le \Gamma$.  It follows that
 \begin{equation} \label{Eq.gammagamma}
 \P \( \sum_{z \in  D} \varphi (z) \ge \Gamma \)  \le \sum_{r=0}^{k-1} \P \( \sum_{z \in D} \varphi_r (z) \ge \Gamma_r \)  .
 \end{equation}
Next, we will fix $r$ and deal with $\P \( \sum_{z \in  D} \varphi_r (z) \ge \Gamma_r \) $ by calculating $\E \( \sum_{z \in  D} \varphi_r (z) \) ^2$. By (i) and \eqref{Eq.ellipseshape}, for each $z$, there are at most $ c_1 2^{r+2}$ vertices $w \in D$ such that $\varphi_r (w)$ is not independent with $\varphi_r (z)$. Hence, $\E \( \varphi_r (z) \sum_{w \in D} \varphi_r (w) \) \le c_1 c_2 2^{r+2} $, by (ii). Combined with \eqref{Eq.ellipseshape}, this yields that
 $$
\s_r^2 : = \E \( \sum_{z \in D} \varphi_r (z) \) ^2 = \sum_{z \in D} \E \(  \varphi_r (z) \sum_{w \in D} \varphi_r (w)  \) \le  c_1 K c_1 c_2 2^{r+2} =  c_1^2 c_2 2^{k+r +2} .
 $$
By straightforward computations for Gaussian variables, we get that for each $r = 0, \ldots, k-1$,
 $$
\P \( \sum_{z \in D} \varphi_r (z) \ge \Gamma_r \) \le 2 e^{- \frac { \Gamma_r^2} {2 \s_r^2} } = 2 \exp \left\{ -  \frac { 2^{- (k-r ) - 6}  \e^2  \d^2 K^2 k^2} {c_1^2 c_2 2^{k+r+3}}  \right\}
 =
2 e^{-  \frac 1 {512 c_1^2 c_2}  \e^2  k^2  \d^2  } .
 $$
Combined with \eqref{Eq.decomde}, \eqref{Eq.tildeEdEd}, \eqref{Eq.gammagamma} and the assumption that $k \ge 6$, this yields the result.
 \end{proof}

\medskip

\noindent {\it Proof of Proposition~\ref{Prop.Lemma1}.} Equivalently, we will show the result for $j+1$, where $0 \le j \le m-2$.

Suppose $ (B_1, B_2) \in \END_{j+1} $. Let
 $$
A = \cup_{P \in T_{j+1} (B_1, B_2)} P, \ \ \ \ \ \hat D = \Z^2 \cap A \ \ \ \mbox{ and } \ \ \ \tilde D = \left\{ l_B : B \in \BD^A_{K^j} \right\} ,
 $$
where $\BD^A_{K^j}$ consists of boxes in $\BD_{K^j}$ intersecting with $A$, and $l_B$ is the lower left corner of $B \cap \Z^2$. By the definition of tame paths, there is a universal constant $\tilde C_1$ such that for each $(B_1, B_2) \in \END_{j+1}$,
 \begin{equation} \label{Eq.ellipssize}
|\hat D| \le \tilde C_1  K^{2j+1},  \ \ \ \hat D_1 \wedge \hat D_2 \le \tilde C_1 K^j , \ \ \ | \tilde D| \le \tilde C_1  K \ \ \ \mbox{and}  \ \ \ \tilde D_1 \wedge \tilde D_2 \le \tilde C_1 .
 \end{equation}
We set $c_1 = \tilde C_1$, $c_2 = C_1 + 2 C_2$ and $c_3 = C_4$ in Lemma~\ref{Lem.Lemma1}. Consequently, $c$ therein is set correspondingly, which is a universal constant here. Then $k \ge (c_3/c_2) \vee 6$ is satisfied, since $K$ is large. Set $C_7 = \min \left\{  \frac 1 {18 C_5} , \frac 1{600} c \right\} $ and $C_8 = 192 c_1 \sqrt {2 c_2}$, which are universal constants. Note that \eqref{Eq.Xindpt}, \eqref{Eq.C1C2}, Lemma~\ref{Lem.positivecorrelated} and \eqref{Eq.ellipssize} respectively imply (i)-(iii) and \eqref{Eq.ellipseshape} in Lemma~\ref{Lem.Lemma1}. That is, the assumptions in Lemma~\ref{Lem.Lemma1} all hold. Therefore, we will apply Lemma~\ref{Lem.Lemma1} without further checking conditions (i)-(iii) and  \eqref{Eq.ellipseshape}.

For $j = 0$, let $D = \hat D$ and  $\varphi_r : = X_r$. In this case, $\varphi = \eta_0$.  By Corollary~\ref{Cor.SLm}, $d_P \ge K /2$ for all $ P \in T_1 (B_1, B_2)$. It follows that
 $$
\P \big( \D_P \ge \d \mbox{ for some } P \in T_1 (B_1, B_2) \big) \le \P \( \left| \big\{ z \in D : \varphi (z) \ge \e k \big\} \right| \ge  \frac 1 2 \d K \) .
 $$
By the choice of $C_8$, $\d \ge \frac {C_8} {\sqrt k \e}$ implies that $\d /2 \ge \frac {8 c_1 \sqrt {2 c_2}}{\sqrt k \e / 2}$. Applying Lemma~\ref{Lem.Lemma1} to $\d/2$ (rather than $\d$), we conclude that
 $$
\P \big( \D_P \ge \d \mbox{ for some } P \in T_1 (B_1, B_2) \big) \le 5 k e^{- c \e^2 k^2 (\d / 2)^2} \le K^2 e^{-C_7 \e^2 k^2 \d^2},
 $$
where in the last inequality we have used the choice of $C_7$.

For $j \ge 1$, let $D =K^{-j}  \tilde D$, and let $\varphi_r (z) : = X_{jk + r} (K^j z)$ for all $z \in D$. In this case,  $\varphi (z) = \eta_j (K^j z)$. Applying Lemma~\ref{Lem.Lemma1} to $\e / 2$ (rather than $\e$), we have that for $\d \ge \frac {8 c_1 \sqrt {2 c_2}}{\sqrt k \e / 2}$,
 $$
\P (\hat  E_\d) \le 5 k e^{- c (\e/2)^2 k^2 \d^2}, \ \ \ \mbox{where } \hat E_\d : = \left\{ \left| \left\{ B \in \BD^A_{K^j} : \eta_j (l_B) \ge \frac 1 2 \e k \right\} \right| \ge \d K \right\} .
 $$

Next, we will estimate the fluctuation of $\eta_j$ within a box $B \in \BD_{K^j}^A$. Let $M_B := \max_{z \in B \cap \Z^2} \big( \eta_j (z) - \eta_j (l_B) \big)$. Then, by Lemma~ \ref{Lem.Harmdiff} and \ref{Lem.maxinbox}, we have $\E M_B \le \sqrt {C_5} C_6 $ and $\max_{z \in B \cap \Z^2} \E \big( \eta_j (z) - \eta_j (l_B)\big)^2 \le C_5$. Set $k_1 (\e) = \lfloor 6 \sqrt {C_5} C_6 / \e \rfloor + 1$. Then, for each $K \ge 2^{k_1 (\e)} = : K_1 (\e)$, we have $\frac 1 6 \e k \ge \E M_B$. By Lemma~\ref{Lem.concentration},
 $$
\P \( M_B > \frac 1 2 \e k \)  \le  \P \( M_B - \E M_B > \frac 1 3 \e k \) \le 2 e^{- \frac 1 {18 C_5}  \e^2 k^2} .
 $$
Since $|\tilde D| \le \tilde C_1 K$ by \eqref{Eq.ellipssize}, there are at most $\tilde C_1 K$ boxes $B \in \BD_{K^j}$ intersecting with $A$. By a union bound,
 $$
\P \( E^M \) \le \tilde C_1 K \times 2 e^{- \frac 1 {18 C_5}  \e^2 k^2}, \ \ \ \mbox{where } E^M : = \left\{ M_B > \frac 1 2 \e k \mbox{ for some } B \in \BD^A_{K^j} \right\} .
 $$
Note that
 $$
\tilde E_\d \subseteq E^M \cup \hat E_\d, \ \mbox{ where } \tilde E_\d : = \left\{ \left| \big\{ B \in \BD^A_{K^j} : \eta_j (z) \ge \e k \mbox{ for some } z \in B \big\} \right| \ge \d K \right\} .
 $$
It follows that
 $$
\P (\tilde E_\d) \le \P (E^M) + \P (\hat E_\d ) \le  2 \tilde C_1 K e^{- \frac 1 {18 C_5} \e^2 k^2 } + 5 k e^{-  \frac 1 4 c \e^2 k^2 \d^2}  \ \ \ \mbox{for all } \d \ge \frac {8 c_1 \sqrt {2 c_2}}{\sqrt k \e / 2} .
 $$

Finally, by (a) and (c) in Proposition~\ref{Prop.sub-paths}, the event $\D_P \ge \d$ for some $P \in T_{j+1} (B_1, B_2)  $ implies that $ \tilde E_{\d / 12}$ occurs. For $\d \ge C_8 / \sqrt k \e$, it holds that $\d / 12 \ge \frac {8 c_1 \sqrt {2 c_2}}{\sqrt k \e / 2}$. Consequently,
 $$
\P \big( \D_P \ge \d \mbox{ for some } P \in T_{j+1} (B_1, B_2) \big) \le 2 \tilde C_1 K e^{- \frac 1 {18 C_5} \e^2 k^2 \d^2} + 5 k e^{- \frac 1 4 c  \e^2 k^2 (\d/12)^2}  \le K^2 e^{- C_7 \e^2 k^2 \d^2} ,
 $$
completing the proof. \qed

\subsection {The fraction of vertices with high Gaussian values}  \label{Sect.Lemma3}

We first recall that $\p_{\k, \d, K} = \left\{ P :  P \mbox{ is a path in } V_N, \ \| P \| \ge \k N \mbox{ and } |P| \le N^{1 + \frac \d {K^2 k}} \right\} $, and that $\p_\k  = \left\{ P : \| P \| \ge \k N \right\} $, as well as $\p_{\k, \d, K} \subseteq \p_\k$. Also recall that a path $P$ is associated with a tree $\t_P$, which has depth $j \in [m-1]$ if $P \in \SL_j$, and depth $m$ if $P \in \p_\k$. Recall that $\theta_P$ is the unit uniform flow on $\t_P$. In addition we recall that $\rho$ and $\mathcal L$ are respectively the root and the set of leaves of $\t_P$, and that $L(u)$ is the level of a node $u$, with $L(\rho) = 0$. Furthermore, recall that a node $u \in \t_P$ is identified with a sub-path $P^u$ of $P$, and $u$ is said to be tame/open if so is $P^u$. Let $\D_u = \D_{P^u}$, which depends on $\eta_{j-1}$ if $P^u \in \SL_j$. In this section, we aim to show the following proposition.

 \begin{prop} \label{Lem.Lemma3}
There exists a large $K(\d, \e) > 0$ such that for each $K \ge K(\d, \e)$,
 $$
\lim_{m \to \infty} \P \( \sum_{v \in \t_P \setminus \{ \rho \} } \theta_P (v) 1_{\{ v \mbox{ is open} \}}  \le 4 \d m \ \mbox{ for all }P \in \p_{\k, \d, K} \) = 1 .
 $$
 \end{prop}

We will express $\sum_{v \in \t_P \setminus \{ \rho \} } \theta_P (v) 1_{\{ v \mbox{ is open} \}} $ in terms of
 \begin{equation} \label{Eq.defYPr}
Y_{P, r} : = \sum_{u \in \t_{P, r} } \theta_P (u) \D_u 1_{\{ u \mbox{ is tame} \}} ,
 \end{equation}
where $\t_{P, r} : = \{ u \in \t_P : L(u) = r \} $. Note that $Y_{P,r}$'s satisfy the following recursion formula:
 \begin{equation} \label{Eq.YPr}
Y_{P, r+1} = \frac 1 d \sum_{i=1}^d \sum_{u \in \t_{ P^{(i)}, r} } \theta_{P^{(i)}}(u) \D_u 1_{\{ u \mbox{ is tame} \}} = \frac 1 d \sum_{i=1}^d Y_{P^{(i)}, r} ,
 \end{equation}
where $d = d_P$ and $P^{(i)}$, $i \in [d]$ are child-paths of $P$ constructed in Proposition~\ref{Prop.sub-paths}. Let
 $$
\xi_{r, j, B_1, B_2} : = \max \left\{ Y_{P, r}, \ P \in \p_j \big( B_1, B_2 \big) \right\} .
 $$

Recall (c) of Proposition~\ref{Prop.sub-paths}. We denote by $\END_{j, d}$ the sequences $(B_{i,1}, B_{i,2})$, $i \in [d]$ in $\END_j$ such that the 12-times-rule holds. Namely, let
 $$
\END_{j, d} : = \left\{ \big\{ \( B_{i,1}, B_{i,2} \) \big\}_{i \in [d]} \subseteq \END_j : \left| \big\{ i : B_{i,1} \subseteq B \big\} \right| \le 12  \  \mbox{ for all } B \in \BD_{K^j} \right\} .
 $$

 \begin{lemma} \label{Lem.inductionxir}
Set
 $$
\b = 1 / 2^9, \ \ \ c_r = C_7 \e^2 (\b K)^r, \ \ \ \d_0 = \frac {C_8 \vee \sqrt {2 / {C_7}  }} {\e \sqrt k}  \ \ \ \mbox{and} \ \ \ \d_r = \d_0 + \sum_{s=0}^{r-1} \frac {4 + ( 1 \vee \sqrt {2 s / 3} ) }{ \sqrt { c_s \b k} } .
 $$
Then, there exists a large $K_2 (\e) > 0$ such that for each $K \ge K_2 (\e)$ the following hold for all $j \in [m-1]$ and $0 \le r \le j-1$.

(i) Suppose $(B_1, B_2) \in \END_j$. Then, $\xi := \xi_{r, j, B_1, B_2} $ satisfies
 $$
\P(\xi > \d) \le 2 e^{- c_r k^2 (\d - \d_r)^2} \ \ \ \mbox{ for all } \d \ge \d_r .
 $$

(ii) Let $d \ge 1$. Suppose $\left\{ \( B_{i,1} , B_{i,2} \) \right\} _{i \in [d] } \in \END_{j,d}$, and let $\xi_i = \xi_{r, j, B_{i,1}, B_{i,2} }$, $i \in [d]$. Then,
 $$
\P \( \frac 1 d \( \xi_1 + \cdots + \xi_d \) > \d \) \le  \( K^{- 16 }  e^{- \b c_r k^2  (\d - \d_{r+1})^2} \) ^d  \ \ \ \mbox{ for all } \d \ge \d_{r+1} .
 $$
 \end{lemma}
  \begin{proof}
Set $k_2 (\e) := \min \{ k \in \Z : \sqrt {C_7} \e k \ge 1 \mbox{ and } 3 \sqrt{C_7} \e k^2 e^{- \frac 1 6 k} \le 1 \} \vee \log_2 K_1 (\e)$, and set $K_2 (\e) = 2^{k_2 (\e)}$. We will prove the results by induction on $r$. In Step 1, we will check (i) for $r = 0$ and all $j \in [m-1]$. In Step 2, we will show that (i) implies (ii). In Step 3, we will show (i) for $r+1$ and all $j \in [r+2, m-1]\cap \mathbb Z$, provided that (ii) holds for all $j \in [r+1, m-1] \cap \Z$. Assuming these, we will obtain the lemma by the induction, completing the proof. Next, we will carry out the details for these three steps.

\medskip

\noindent {\bf Step 1.} Let $r=0$.  Note that $Y_{P, 0} = \D_P 1_{\{ P \mbox{ is tame} \}}$. Hence, $\xi = \max \left\{ \D_P, P \in T_j \( B_1, B_2 \) \right\} $. We need to check (i) for all $j \in [m-1]$.  This follows from Proposition~\ref{Prop.Lemma1} directly, since $\d \ge \d_0$ implies $K^2 e^{- C_7 \e^2 k^2 \d^2} \le e^{2 k - C_7 \e^2 k^2 \d_0^2 - C_7 \e^2 k^2 (\d - \d_0)^2 } \le 2 e^{- c_0 k^2 (\d - \d_0)^2}$ for $K \ge K_2 (\e) \ge K_1 (\e)$.

\medskip

\noindent {\bf Step 2.} Assuming (i) holds, we will show (ii). For each $P \in \SL_j$, $P \subseteq B(x_P, 2 K^j)$ by the definition of $\SL_j$. If $d_\infty (B_{i,1}, B_{i^\prime,1} ) \ge 5 K^j$, we have $d_\infty \( \cup_{P \in \p_j \( B_{i,1}, B_{i,2} \) } P  , \ \cup_{ P' \in \p_j ( B_{i^\prime,1}, B_{i^\prime,2} ) } P' \) \ge K^j $. Consequently, $\xi_i$ and $\xi_{i^\prime}$ are mutually independent by \eqref{Eq.Xindpt} since they rely on the field $\eta_{j-1-r}$.

Next, we will classify $\xi_i$'s into $432$ groups, such that $\xi_i$'s in each group are mutually independent. Concretely, we classify $\BD_{K^j}$ into $36$ families $\F_s$, $s \in [36]$, where $\F_1$ consists of boxes respectively containing $(6 a K^j,  6 b K^j)$, $a,b \in \Z$ and other $\F_s$'s are its shifts. Accordingly, let $\G_s : = \{ B_{i,1} : i \in [d], \mbox{ and } B_{i,1} \subseteq B \mbox{ for some } B \in \F_s \}$. By the 12-times-rule, we can classify each $\G_s$ into 12 groups $\G_{s,t}$, $t \in [12]$, such that for each $s,t$, a box in $\F_s$ contains at most one $B_{i,1}$ in $\G_{s,t}$. Finally, we classify $B_{i,1}$'s (equivalently,  $\xi_i$'s) into $432 (\le 2^9 = 1/\b )$ groups $\G_{s,t}$'s such that all $\xi_i$'s in each $\G_{s,t}$ are mutually independent.

Set $W_{s,t} : = \prod_{B_{i,1} \in \G_{s,t}} e^{a \b (\xi_i - \d )}$, where $a > 0$. Then
 \begin{eqnarray*}
\E e^{a \b \sum_i (\xi_i - \d) }
 & = &
\E \prod_{s=1}^{36} \prod_{t=1}^{12} W_{s,t} \le \prod_{s=1}^{36} \prod_{t=1}^{12}  \(  \E W_{s,t}^{1/ \b} \) ^\b
 \\ & \le &
\prod_{s=1}^{36} \prod_{t=1}^{12} \prod_{B_{i,1} \in \G_{s,t}}  \( \E e^{ a  ( \xi_i   - \d  ) } \) ^\beta  = \prod_{i=1}^d \( \E e^{a ( \xi_i - \d ) }  \) ^{\beta} .
 \end{eqnarray*}
It follows that for every $a > 0$,
 \begin{equation} \label{Eq.average1}
\P \( \frac 1 d \( \xi_1 + \cdots + \xi_d \) > \d \) \le \E e^{a \b \sum_i (\xi_i - \d) }  \le \prod_{i=1}^d \(  \E e^{a ( \xi_i - \d ) } \) ^\b .
 \end{equation}

Set $\xi = \xi_i$ for any $i$ for brevity. Next, we will estimate $\E e^{a \xi }$. By (i) and the fact that $0 \le \xi \le 1$, it holds that
 \begin{eqnarray*}
\E e^{a \xi}
 & = &
e^{a \d_r} + \int_{\d_r}^1 \P( \xi > z) a e^{a z} d z \le e^{a \d_r} + 2 a e^{a \d_r} \int_{-\infty}^{\infty} e^{- c_r k^2 (z-\d_r)^2} e^{a (z-\d_r)} d z
 \\ & \le &
e^{a \d_r} + \frac a { \sqrt {c_r}} e^{a \d_r + \frac {a^2}{4 c_r k^2}} \le \( 1 + \frac a { \sqrt {c_r}} \) e^{\frac {a^2}{4 c_r k^2} + a \d_r }.
 \end{eqnarray*}
Consequently,  $\E e^{a ( \xi - \d )}  \le \( 1 + \frac a { \sqrt {c_r}} \) e^{\frac {a^2}{4 c_r k^2} - a (\d - \d_r) } $. For each $\d \ge \d_r$, we set $a = 2 c_r k^2 (\d - \d_r)$ so that the exponent is optimized. Then,
 $$
\E e^{a ( \xi  - \d )}  \le \( 1 + \frac {2 c_r k^2 (\d - \d_r)} { \sqrt {c_r}} \) e^{  - c_r k^2 (\d - \d_r)^2 } \le  3 \sqrt {c_r} k^2 e^{- c_r k^2 (\d - \d_r)^2 } ,
 $$
where in the last inequality we have used $K \ge K_2 (\e)$.

Note that
 $$
\d_{r+1} - \d_r = a_r + b_r, \ \ \ \mbox{where } a_r = \frac {4} {\sqrt { c_r \b k} } \mbox{ and }  b_r =  \frac {1 \vee \sqrt {2 r / 3} }{ \sqrt { c_r \b k} } ,
 $$
 $$
e^{- \b c_r k^2 a_r^2} = e^{- 16 k} \le K^{- 16},
  $$
  $$
\mbox{and } \ \ \ 3 \sqrt {c_r} k^2  e^{- \b c_r k^2 b_r^2}  = 3 \sqrt{C_7} \e k^2 (\b K)^{\frac r 2} e^{ - (1 \vee \frac {2r}{3}) k } \le 3 \sqrt{C_7} \e k^2 e^{- \frac 1 6 k} \le 1,
 $$
where in the last inequality we have used $K \ge K_2 (\e)$.

Combining the four displays above, we conclude that for $\d \ge \d_{r+1}$,
 $$
\( \E e^{a ( \xi - \d )}  \) ^ \b
 \le
3 \sqrt {c_r} k^2   e^{- \b c_r k^2 \(  (\d - \d _{r+1} )^2 + a_r^2 + b_r^2 \) }  \le K^{-16}  e^{- \b c_r k^2  (\d - \d _{r+1} )^2 } .
 $$
The above inequality holds for $\xi = \xi_i$ for all $i$. Combined with  \eqref{Eq.average1}, this implies (ii).

\medskip

\noindent {\bf Step 3.} Assume that (ii) holds for all $j \in [r+1, m-1] \cap \Z$. Then, we will show (i) for $r+1$ and all $j \in [r+2, m-1] \cap \Z$. Suppose that $(B_1, B_2) \in \END_j$. First, we fix the end-boxes of the child-paths. Concretely, for $d \ge K$ and each sequence $\SS : = \left\{ \big( B_{i,1}, B_{i,2} \big) \right\}_{i \in [d]}$ in $\END_{j-1, d}$, we set
 $$
\p_{j,d} : =  \left\{ P \in \p_j \big( B_1, B_2 \big) : d_P = d \right\} , \ \ \ \zeta_d : =  \max \big\{ Y_{P, r} \ , \ P \in \p_{j,d} \big\} ,
 $$
 $$
\p_{j, \SS } : = \left\{ P \in \p_{j,d} : \ P^{(i)} \in \p_{j-1} \big( B_{i,1}, B_{i,2} \big) \mbox{ for all } i \in [d] \right\} , \ \ \
\zeta_{\SS} : = \max \big\{ Y_{P, r} \ , \ P \in \p_{j, \SS } \big\} .
 $$
By \eqref{Eq.YPr}, $ \zeta_{\SS} \le \frac 1 d \sum_{i=1}^d \xi_i$, where $\xi_i : = \xi_{r, j-1, B_{i,1}, B_{i,2}}$ for $i \in [d]$. If $r+ 2 \le j \le m-1$, we have $r+1 \le j-1 \le m-1$. For each $\d \ge \d_{r+1}$, we apply (ii) to $j-1$, and have
 $$
\P (\zeta_{\SS} > \d )  \le  \( K^{- 16}  e^{- \b c_r k^2  (\d - \d _{r+1} )^2 }  \) ^d  \ \ \ \mbox{ for all } \SS \in \END_{j-1, d}.
 $$

By the definition of $\END_{j-1}$ and $\SL_j$, we have that $B_{i,1}, B_{i,2} \in \BD_{K^{j-3}}$ and that there are at most $K^7$ boxes in $\BD_{K^{j-3}}$ intersecting with some path in $\p_j (B_1, B_2)$. Therefore, we can find at most $K^{14 d}$ sequences $\SS_i \in \END_{j-1, d}$ such that $\p_{j,d}  \subseteq \cup_i \p_{j, \SS_i }$. By a union bound,
 \begin{equation} \label{Eq.YPr11}
\P (\zeta_d > \d) \le K^{14 d} \max_i \P (\zeta_{\SS_i} > \d) \le  \( K^{- 2}  e^{- \b c_r k^2  (\d - \d _{r+1} )^2 }  \) ^d .
 \end{equation}
Note that $ K^{- 2}  e^{- \b c_r k^2  (\d - \d _{r+1} )^2 }  \le K^{-2} \le \frac 1 2$. By (a) of Proposition~\ref{Prop.sub-paths},
 $$
\P \( \xi > \d \) \le \sum_{d = K}^\infty \P (\zeta_d > \d) \le 2 \( K^{- 2}  e^{- \b c_r k^2  (\d - \d _{r+1} )^2 }  \) ^K \le 2  e^{- \b K c_r k^2  (\d - \d _{r+1} )^2 }  =   2  e^{- c_{r+1} k^2  (\d - \d _{r+1} )^2 } .
 $$
That is, (i) holds for $r+1$ and all $j \in [r+2, m-1] \cap \Z$.
  \end{proof}

\medskip

\noindent  {\it Proof of Proposition~\ref{Lem.Lemma3}.}
First, we write $\sum_{v \in \t_P \setminus \{ \rho \} } \theta_P (v) 1_{\{ v \mbox{ is open} \}} $ in terms of $Y_{P,r}$'s as follows.
 \begin{eqnarray*}
  & &
\sum_{v \in \t_P \setminus \{ \rho \}} \theta_P(v) 1_{\{ v \mbox{ is open} \}}  = \sum_{u \in \t_P \setminus \l } \sum_{v} \theta_P(v) 1_{\{ v \mbox{ is an open child of } u \}}  = \sum_{u \in \t_P \setminus \l } \theta_P(u)  \D_u
  \\ & \le &
\theta_P (\rho) +  \sum_{u: 1 \le L(u) \le m-1 } \theta_P(u) \D_u 1_{\{ u \mbox{ is tame} \}} +  \sum_{u: 1 \le L(u) \le m-1} \theta_P(u) 1_{\{ u \mbox{ is untamed} \}} ,
 \end{eqnarray*}
where we have used $\D_u \le 1$ for all $u$. Recall \eqref{Eq.defYPr}. By Proposition~\ref{Lem.Lemma2} and the fact that $\theta_P (\rho) = 1 \le \d m$, we have
 $$
\sum_{v \in \t_P \setminus \{ \rho \}} \theta_P(v) 1_{\{ v \mbox{ is open} \}}  \le \sum_{r = 1}^{m - 1 } Y_{P,r} + 3 \d m \ \ \ \mbox{for all } P \in \p_{\k , \d, K} .
 $$
Therefore, in order to prove Proposition~\ref{Lem.Lemma3} we only need to check
 \begin{equation} \label{Eq.star}
\lim_{m \to \infty} \P \( \sum_{r=1}^{m-1} Y_{P,r} > \d m \mbox{ for some } P \in \p_{\k, \d, K} \) = 0 .
 \end{equation}

Next, we check \eqref{Eq.star}. Recall that $c_r = C_7 \e^2 (\b K)^r$, where $\b = 1/ 2^9$. The definition of $\d_r$ in Lemma~\ref{Lem.inductionxir} implies that $\d_r \le \frac 1 \e \d (K)$ for all $r \ge 0 $, with
 $$
\d (K) := \frac 1 {\sqrt k} \(  C_8 \vee \sqrt { \frac 2 {C_7}  } + \sum_{s=0}^\infty \frac {4 + ( 1 \vee \sqrt {2 s / 3} ) }{ \sqrt {  C_7  \b (\b K)^s} } \) .
 $$
There exists $K(\d, \e) \ge K_2 (\e)$ such that $\d \big( K(\d, \e) \big) < \frac 1 2 \d \e$. For each $K \ge K(\d, \e)$, we have $\sum_{r=0}^{m-1} ( \d_r + \frac 1 4 2^{-r} m \d ) \le m \frac 1 2 \d + \frac 1 2 m \d = m \d$. Consequently, to prove \eqref{Eq.star}, it suffices to check that $\lim_{m \to \infty} \sum_{r=1}^{m-1} p_{m,r} = 0$, where
 $$
p_{m,r} := \P \( Y_{P, r} > \d_r + \frac 1 4 2^{-r} m \d  \mbox{ for some } P  \in \p_{\k, \d, K}   \) .
 $$
We now estimate $p_{m,r}$, following the arguments in Step 3 in the proof of Lemma~\ref{Lem.inductionxir}. Since $\k N < K^{m+2}$, there are at most $\( K^6 / \k \) ^2$ boxes in $\BD_{m-3}$ intersecting with $V_N$. Recall that $d_P = d_0$ by Corollary~\ref{Cor.Pktree}. Similar to \eqref{Eq.YPr11}, we have
 $$
p_{m,r} \le \( \frac {K^{24}}{\k^4} \times K^{-16} e^{- \b c_{r-1} k^2  (\frac 1 4 2^{-r} m \d )^2  } \) ^{d_0} = \( \frac {K^{8}}{\k^4} \exp \left\{ - 2 a \( \frac {\b K}{4} \) ^{r-1}  m^2  \right\} \) ^{d_0} ,
 $$
where
 $$
a =  \frac 1{128}  C_7 \e^2 \b k^2 \d^2 .
 $$
Note that for $r \in [m-1]$,
 $$
\frac {K^8} {\k^4} \exp \left\{ - a \( \frac {\b K}{4} \) ^{r-1}  m^2  \right\} \le \frac {K^8} {\k^4} e^{- a m^2}   \le 1, \ \ \ \mbox{ and } \exp \left\{ - a \( \frac {\b K}{4} \) ^{r-1}  m^2  \right\}  \le e^{- a m^2 r}
 $$
since $m \to \infty$, and $\b K / 4 \ge e$ implying that $(\b K / 4)^{r-1} \ge r$ for all $r \ge 1$. It follows that
 $$
p_{m,r} \le  e^{- a m^2 r d_0} \le e^{- a K m^2 r}  \ \ \ \mbox{ for all } r \in [m-1],
 $$
noting $d_0 \ge K$ by \eqref{Eq.d0}. Consequently, $\sum_{r=1}^{m-1} p_{m,r} \le \frac 1 {e^{a K m^2} - 1}$, which converges to $0$ as $m \to \infty$. Therefore, \eqref{Eq.star} holds.   \qed

\subsection{Proof of Theorem~\ref{mainthm}}  \label{Sect.Pfmain}

Recall that each $P \in \p_{\k, \d, K}$ is associated with a tree with depth $m$, and that the set of leaves is denoted by $\mathcal L$. In what follows, we set $\e = \frac 1 2 \sqrt \d $. That is, a path $P$ in $\SL_j$ is said to be open if $\eta_j (z) \ge \frac 1 2 \sqrt \d k$ for some $z \in P \cap \Z^2$. For a leaf $v$, denote by $O_v$ the number of open ancestors of $v$ (including itself, excluding the root). We say that $v$ is heavy if $O_v \ge 8 \d m$, and light otherwise. Define
 $$
E_1 := \left\{ \sum_{v \in \l } \theta_P (v) 1_{ \{ v \mbox{ is heavy} \} } \le \frac1 2  \ \ \ \mbox{for all } P \in \p_{\k, \d, K} \right\} ,
 $$
 $$
E_2 : = \left\{ \sum_{i=1}^r \eta_{j_i} (z) \le \sqrt {320 C_1 \d} \log N \ \mbox{ for all } z \in V_N, r \le 8 \d m \mbox{ and } 0 \le j_1 < \cdots < j_r \le m-1 \right\} ,
 $$
 $$
E_3 := \left\{ \max_{u \in V_N} H_{K^m}  (u) \le \frac 1 2 \sqrt \d \log N \right\} .
 $$

On the one hand, suppose that $E_1$ occurs. By Corollary~\ref{Cor.Pktree} and (a1), (a2) of Proposition~\ref{Prop.tree}, $\theta_P (v) \le \frac 2 {K^{m-1} d_0 } $ for each $v \in \l $. Then, we can find at least $\frac 1 2 / \frac 2 {K^{m-1} d_0} = \frac {d_0 K^{m-1}} 4 \ge \frac 1 8 \k N$ light leaves, for each $P \in \p_{\k, \d ,K}$.

On the other hand, suppose that $E_2$ and $E_3$ occur. We call $z$ a good vertex if $\eta^{V_{3N}}  (z) \le C \sqrt \d \log N$, where $C : = \sqrt {320 C_1} + 2 $. Then, for each light leaf $v$, $z := P^v$ is a good vertex since
 \begin{eqnarray*}
\eta^{V_{3N}} (z)  & = & \eta (z) + H_{K^m} (z) \le \sum_{j =0}^{m-1} \eta_j (z) 1_{\{ \eta_j (z) \ge  \frac 1 2 \sqrt \d  k \} } +   \sum_{j =0}^{m-1} \eta_j (z) 1_{ \{ \eta_j (z) <  \frac 1 2 \sqrt \d k \} }  +  \frac 1 2 \sqrt \d  \log N
 \\ & \le &
\sqrt {320 C_1 \d} \log N +  m  \frac 1 2 \sqrt \d  k +  \frac 1 2 \sqrt \d  \log N =  (\sqrt {320 C_1 \d}  + \frac 3 2 \sqrt \d ) \log N \le  C \sqrt \d \log N .
 \end{eqnarray*}

Therefore, on $E_1 \cap E_2 \cap E_3$, there are at least $\frac 1 8 \k N$ good vertices on $P$ for all $P \in \p_{\k, \d, K}$. Later, we will show that \\
(i) $\lim_{m \to \infty} \P (E_1) = 1$ for all $K \ge K \( \d, \frac 1 2 \sqrt \d \) $, where $K(\d, \e)$ is given in Proposition~\ref{Lem.Lemma3};\\
 (ii) $\lim_{m \to \infty} \P (E_2) = 1$ for all $K \ge K_3 (\d)$, where $K_3 (\d)$ is given in Lemma~\ref{Lem.tildeK} below;\\
 (iii) $\lim_{m \to \infty} \P (E_3) = 1$.

  Assuming (i), (ii) and (iii), we will conclude that $\lim_{N \to \infty} P(\tilde E) = 1$ by symmetry, where
 $$
\tilde E : = \left\{  \left| \left\{ z \in P : \eta^{V_{3N}} (z) \ge -  C \sqrt \d \log N \right\} \right| \ge \frac 1 {8 } \k N \ \ \mbox{ for all } P \in \p_{\k, \d, K} \right\} .
 $$
By Remark~\ref{Rem.aboutC1}, $C_1 \le 1/2$. Thus $C \le \sqrt {320 / 2} + 2 = 15$. Setting
 $$
K (\delta) =  K \( \delta, \frac 1 2 \sqrt \d \) \vee K_3 (\delta) \ \ \ \mbox{ and } \ \ \ \a (\delta) = \frac {\d}{K(\delta)^2 k(\delta)} ,
 $$
we will conclude the proof of Theorem~\ref{mainthm}. Thus, it remains to check (i), (ii) and (iii).

 \begin{lemma}
Let $K (\d, \e)$ be as in Proposition~\ref{Lem.Lemma3}. For each $K \ge K \( \d, \frac 1 2 \sqrt \d \) $, $\lim_{m \to \infty} \P (E_1) = 1$.
 \end{lemma}
 \begin{proof}
Note that
 \begin{eqnarray*}
 & &
\sum_{u \in \t \setminus \{ \rho \} } \theta_P (u) 1_{\{ u \mbox{ is open}  \}} = \sum_{u \in \t \setminus \{ \rho \} } \sum_{v \in \l } \theta_P (v) 1_{\{  u \mbox{ is an open ancestor of } v \}}
 \\ & = &
\sum_{v \in \l } \theta_P (v) O_v \ge 8 \d m \sum_{v \in \l } \theta_P (v) 1_{\{ v \mbox{ is heavy} \}} .
 \end{eqnarray*}
Consequently, $E_1$ occurs if $   \sum_{u \in \t \setminus \{  \rho \} } \theta_P (u) 1_{\{ u \mbox{ is open} \}} \le 4 \d m$ for all $P \in \p_{\k, \d, K}$. By Proposition~\ref{Lem.Lemma3}, the result holds.
 \end{proof}

 \begin{lemma} \label{Lem.tildeK}
There exists a large $K_3 (\d) > 0$ such that for each $K \ge K_3(\d)$, $\lim_{m \to \infty} \P(E_2) = 1$.
 \end{lemma}
 \begin{proof}
Let $G(z) = \sum_{i=1}^r \eta_{j_i} (z) $. Note that $\eta_j (z)$'s are mutually independent. By \eqref{Eq.C1C2},  $\E G(z)^2 = \sum_{i=1}^r \E \eta_{j_i} (z)^2 \le r (C_1 k + 2 C_2) \le 16 C_1 \d \log N$, where the last inequality holds for $r \le 8 \d m$. Thus,
 $$
\P \( G(z) >  \sqrt {320 C_1 \d} \log N \) \le 2 e^{- \frac 1 {32 C_1 \d \log N} 320 C_1 \d  \log^2 N} = 2 N^{-10}\,.
 $$
A union bound implies $\P(G(z) >  \sqrt {320 C_1 \d} \log N \mbox{ for some } z \in V_N ) \le  N^{- 7}$. It follows that
 $$
\P(E_2^c) \le \sum_{r \le 8 \d m } \frac {m!}{r! (m-r)!} \times N^{- 7} \le  a_\d \sqrt m \( b_\d K^{-7} \) ^m ,
 $$
where $a_\d$ and $b_\d$ are constants depending on $\d$. Pick $K_3 (\d)$ such that $b_\d K_3 (\d)^{-7} < 1$. Then, for each $K \ge K_3 (\delta)$, the right hand side above converges to $0$ as $m \to \infty$.
 \end{proof}

 \begin{lemma}
$\lim_{m \to \infty} \P(E_3) = 1$.
 \end{lemma}
 \begin{proof}
Denote by $B_i$'s the boxes in $\BD_{K^m}$ intersecting with $V_N$. For any $i$, we have $u \in V_{3N}^{(\frac 1 {10})}$ and $B_{K^m} (u) \subseteq V_{3N}$ for all $u \in B_i \cap V_N$. Recall that $\k N < K^{m+2}$. Then, by \eqref{Eq.logcorrelated},
 $$
\E H_{K^m} (u)^2 = \E \eta^{V_{3N}} (u)^2 - \E \eta^{B_{K^m} (u)} (u)^2 \le C_1 \log_2 \frac {3 N}{K^m} + 2 C_2 \le C_1 \( 3 k - \log_2 \k \)
 $$
for all $u \in B_i \cap V_N$. Let $M_i : = \max_{u \in B_i \cap V_N} H_{K^m} (u)$. By Lemmas~\ref{Lem.Harmdiff} and \ref{Lem.maxinbox}, we get $\E M_i \le C_6 \sqrt {C_5} $ for all $i$. Consequently, by Lemma~\ref{Lem.concentration}, it holds that
 $$
\P \( M_i >  \frac 1 2 \sqrt \d  \log N \) \le \P \( M_i - \E M_i >  \frac 1 4 \sqrt \d  \log N  \)  \le 2 e^{ - \frac 1 { 32 C_1 ( 3k - \log_2 \k )  } \d \log^2 N  }  \ \ \mbox{ for all } i.
 $$
Since $\k N < K^{m+2}$, there are at most $\( K^3 / \k \) ^2$ such $B_i$'s. A union bound implies that $\P (E_3^c) \le \( K^3 / \k \) ^2 2 e^{ -  \frac 1 { 32 C_1 ( 3k - \log_2 \k )  } \d \log^2 N  } $. Therefore, $\lim_{m \to \infty} \P (E_3) = 1$.
 \end{proof}


\end{document}